\documentclass[oneside,english]{amsart}
\usepackage[T1]{fontenc}
\usepackage{geometry}
\usepackage{color}
\usepackage{mathrsfs}
\usepackage{amstext}
\usepackage{amsthm}
\usepackage{amssymb}

\makeatletter
\numberwithin{equation}{section}
\numberwithin{figure}{section}
\theoremstyle{plain}
\newtheorem{thm}{\protect\theoremname}[section]
\theoremstyle{remark}
\newtheorem{rem}[thm]{\protect\remarkname}
\theoremstyle{plain}
\newtheorem{prop}[thm]{\protect\propositionname}
\theoremstyle{plain}
\newtheorem{lem}[thm]{\protect\lemmaname}
\ifx\proof\undefined
\newenvironment{proof}[1][\protect\proofname]{\par
	\normalfont\topsep6\p@\@plus6\p@\relax
	\trivlist
	\itemindent\parindent
	\item[\hskip\labelsep\scshape #1]\ignorespaces
}{%
	\endtrivlist\@endpefalse
}
\providecommand{\proofname}{Proof}
\fi

\makeatother

\usepackage{babel}
\providecommand{\lemmaname}{Lemma}
\providecommand{\propositionname}{Proposition}
\providecommand{\remarkname}{Remark}
\providecommand{\theoremname}{Theorem}

\begin{document}
\title{On blow up for a class of radial Hartree type equations}
\author{Shumao Wang }
\address{Academy of Mathematics and Systems Science, the Chinese Academy of Sciences, Beijing 100190, China.}
\email{wangshumao@amss.ac.cn}
\begin{abstract}
We study a class of Hartree type equations and prove a quantitative
blow up rate for their blow up solutions. This is an analogue of the
result by Merle and Rapha\"el on 3d NLS. 
\end{abstract}


\maketitle

\section{INTRODUCTION}

We consider the Cauchy problem for the focusing Hartree equation for
$d=3$:
\begin{equation}
\begin{cases}
i\partial_{t}u+\triangle u=-(V*|u|^{2})u, & (t,x)\in\mathbb{R}\times\mathbb{R}^{3},\\
u\mid_{t=0}=u_{0}\in\dot{H}^{\frac{1}{2}}\cap\dot{H}^{1}.
\end{cases}\label{eq:Hartree}
\end{equation}

Here $V(x)$ is a real valued function, problem (\ref{eq:Hartree})
has three conservation laws:
\begin{itemize}
\item Mass:
\begin{equation}
M(u(t))=M(u_{0}):=\int_{R^{3}}|u(t)|^{2}dx,\label{eq:mass}
\end{equation}
\item Energy:
\begin{equation}
E(u(t))=E(u_{0}):=\frac{1}{2}\int|\nabla u(t,x)|^{2}dx-\frac{1}{4}\int(V*|u|^{2})(t,x)|u(t,x)|^{2}dx,\label{eq:energy-1}
\end{equation}
\item Momentum:
\begin{equation}
P(u(t))=P(u_{0}):=Im\int\bar{u}\nabla u(t)dx.\label{eq:momentum}
\end{equation}
\end{itemize}
We will focus on radial $V$ and radial initial data $u_{0}$ throughout
the article. 

\subsection{Setting of the problem and statement of the main result}

For the equation (\ref{eq:Hartree}), it describes the dynamics of
the mean-field limits of many-body quantum systems, such as coherent
states, condensates. In particular, it provides effective model for
quantum systems with long- range interactions. Readers can refer to
\cite{phy-1,phy-2,key-14} for more information about the physical
background of the equation. Besides its physical importance, a lot
of mathematical interest for equation (\ref{eq:Hartree}) lies in
its connection and similarity to cubic NLS,
\begin{equation}
\begin{cases}
i\partial_{t}u+\triangle u=-|u|^{2}u. & (t,x)\in\mathbb{R}\times\mathbb{R}^{3},\\
u\mid_{t=0}=u_{0}.
\end{cases}\label{eq:Schrodinger}
\end{equation}

In particular, when $V(x)=\delta(x)$, then (\ref{eq:Hartree}) formally
becomes (\ref{eq:Schrodinger}). From this perspective, if one obtains
some result for (\ref{eq:Schrodinger}), it is reasonable to expect
that a similar result holds for (\ref{eq:Hartree}). It is also well-expected
that one can generalize a result to a cubic NLS from a Hartree type
model. And this is the main purpose of the current article. However,
the Hartree equation differs from cubic NLS mainly in two ways,
\begin{itemize}
\item Equation (\ref{eq:Hartree}) does not necessarily enjoy the scaling
symmetry\footnote{If $u(t,x)$ solves (\ref{eq:Schrodinger}), so does $u_{\lambda}(t,x):=\frac{1}{\lambda}u(\frac{t}{\lambda^{2}},\frac{x}{\lambda})$.},
which is one of the most important property for (\ref{eq:Schrodinger}).
\item Due to the convolution structure in (\ref{eq:Hartree}), it is non-local.
\end{itemize}
The study of blow up problem for focusing NLS (and other nonlinear
dispersive PDEs) has been an active research field. Classical virial
argument by Glassey in \cite{Glassey} implies the existence of many
blow up solutions. Constructive blow up solutions and universality
of blow up solutions under certain regime have attracted a lot of
researchers, see for examples \cite{key-8,key-9,key-10,key-15,key-21}.
But very few can be said about general blow up solutions, i.e. solutions
without a size constraint, and it is very hard. In \cite{Main reference of FM and PR},
Merle and Rapha\"el consider the nonlinear Schr\"odinger equation
(\ref{eq:Schrodinger}), and proved that,
\begin{thm}
\label{thm PR and FM}Let $u_{0}\in\dot{H}^{\frac{1}{2}}\cap\dot{H}^{1}$
be radial. Assume that the corresponding solution to (\ref{eq:Schrodinger})
blows up in finite time $T$, then there exists a constant $\gamma>0$
such that \footnote{They obtain (\ref{eq:rate}) for $\gamma=\frac{1}{12}$.}
\begin{equation}
||u(t)||_{L^{3}}\geq C(u_{0})|\log(T-t)|^{\gamma}\label{eq:rate}
\end{equation}
\end{thm}
for $t$ close enough to $T$. 
\begin{rem}
\label{rmk 1 of PR and FM}Merle and Rapha\"el have proved a more
general result for a larger class of Schr\"odinger equations. Indeed
they cover all the mass-supercritical and energy-subcritical cases
for $d\geq3$.
\end{rem}
\begin{rem}
\label{rmk 2 of Pr and FM} $\dot{H}^{s_{c}}$ with $s_{c}=\frac{1}{2}$
is the critical norm of (\ref{eq:Schrodinger}), i.e. this norm is
invariant under the natural scaling of (\ref{eq:Schrodinger}). Note
that one has, via Sobolev embedding, $\dot{H}^{\frac{1}{2}}(\mathbb{R}^{3})\hookrightarrow L^{3}(\mathbb{R}^{3})$. 

When the scaling index $s_{c}=0$, we say it is mass-critical; $s_{c}=1$
corresponds to the energy-critical case; in this article we focus on
the case with $0<s_{c}<1$, i.e. mass supercritical and energy subcritical
case.
\end{rem}
\begin{rem}
The estimate (\ref{eq:rate}) may be rephrased as: there is no type
II\footnote{a type II blow up solution means its critical norm remains bounded
when $t$ tends to the blow up time $T$, i.e. $\limsup_{t\nearrow T}||u(t)||_{\dot{H}^{s_{c}}}<+\infty$,
and a blow up solution is type I if it is not type II.} blow up solution to (\ref{eq:Schrodinger}), for radial initial data
in $\dot{H}^{\frac{1}{2}}\cap\dot{H}^{1}$. 
\end{rem}
Thus, all such finite time blow up solutions do not fall into the
regime of Soliton Resolution Conjecture, which formally predicts all
type II blow up solutions to nonlinear dispersive equation will decouple
into solitary wave living at different scales and a regular term. 

The main result of the current article is to prove an analogous result
for a class of Hartree type equations.
\begin{thm}
\label{thm 2 of ours}Consider the equation
\begin{equation}
\begin{cases}
i\partial_{t}u+\triangle u=-(V*|u|^{2})u. & (t,x)\in\mathbb{R}\times\mathbb{R}^{3},\\
u\mid_{t=0}=u_{0}\in\dot{H}^{\frac{1}{2}}\cap\dot{H}^{1}.
\end{cases}\label{eq:Hartree-2}
\end{equation}

where $u_{0}$ is a radial function. Assume $V(x)$ is a radial function
satisfying
\begin{equation}
\sum x_{j}\partial_{x_{j}}V(x)\leq-\alpha V(x)\label{connection condition}
\end{equation}
for some $\alpha\in(2,+\infty)$, and assume the integrability condition
\begin{equation}
V(x)\in L^{1}(\mathbb{R}^{3}),\;\sum x_{j}\partial_{x_{j}}V(x)\in L^{1}(\mathbb{R}^{3}),\label{the integrability condition}
\end{equation}
 and the following pointwise bound
\begin{equation}
|\sum x_{j}\partial_{x_{j}}V(x)|\leq\frac{C}{|x|^{3}}.\label{eq:technisch reason}
\end{equation}
If $u(t)$ blows up in finite time $T$ , then there exists a constant
$\gamma>0$ such that \footnote{Note that (\ref{eq:rate-2}) implies 
\[
||u(t)||_{\dot{H}^{\frac{1}{2}}}\geq C(u_{0})|\log(T-t)|^{\gamma}
\]
 in $\dot{H}^{\frac{1}{2}}$ setting and our method could give $\gamma=\frac{2}{31+\varepsilon}$,
$\forall\varepsilon>0$.}
\begin{equation}
||u(t)||_{L^{3}}\geq C(u_{0})|\log(T-t)|^{\gamma}\label{eq:rate-2}
\end{equation}

for $t$ close enough to $T$. 
\end{thm}
\begin{rem}
\label{rmk 1 of thm 1}Assumption (\ref{the integrability condition})
seems natural. Assumptions (\ref{connection condition}) and (\ref{eq:technisch reason})
are due to technical reasons.
\end{rem}
\begin{rem}
In order to connect virial identity with energy, we have assumed the
condition (\ref{connection condition}) holds. From (\ref{connection condition}),
we note that either 
\begin{equation}
V(x)\geq\frac{c_{V}}{|x|^{\alpha}}\label{eq:1.15}
\end{equation}

at a neighborhood of the origin or $V(x)$ is non-positive. Note that
$V(x)$ can not be non-positive because in this case there exist no
blow up solutions due to energy conservation law.
\end{rem}
\begin{rem}
\label{rmk 2 of thm 1}Under the assumption (\ref{connection condition}),
by the following virial identity ,
\begin{equation}
\frac{d^{2}}{dt^{2}}\int|x|^{2}|u|^{2}=8\int|\nabla u|^{2}+4\int x_{j}|u|^{2}\partial_{j}(V*|u|^{2})dx,\label{virial identity in Hartree seting}
\end{equation}

one could easily derive blow up solutions with negative energy. For
the completeness of the article, we have supplemented the proof of
this part in Appendix A.
\end{rem}
\begin{rem}
\label{the removable condition}The assumptions on $V(x)$ can be
relaxed. Indeed, (\ref{connection condition}) and $V(x)\in L^{1}(\mathbb{R}^{3})$
imply $\sum x_{j}\partial_{x_{j}}V(x)\in L^{1}(\mathbb{R}^{3})$.
\end{rem}
\begin{rem}
\label{rmk 1 of thm 2}We give an example which satisfies the conditions
(\ref{connection condition}), (\ref{the integrability condition}),
and (\ref{eq:technisch reason}) i.e. the object of the above analysis
is not $\textrm{\O}$, 
\begin{equation}
V(Z)=\frac{1}{|Z|^{3}|\log Z|^{\alpha}}\chi(Z),\;for\;\alpha>1,\label{a special example}
\end{equation}

with
\[
\chi(Z)=\begin{cases}
1 & Z\leq\delta\\
0 & Z\geq2\delta
\end{cases}for\;some\;\delta\;small\;enough.
\]
\end{rem}
We review a series of related work regarding Hartree type equations
and NLS,
\begin{itemize}
\item mass-critical case
\end{itemize}
In the 1990s, Merle's work \cite{key-15} had given the characteristic
in $H^{1}$ for the blow-up solutions with minimal mass in nonlinear
Schr\"odinger equation\footnote{The ground state associated to (\ref{mass criticla NLS}) is the unique
positive solution to 
\[
\triangle Q+Q^{1+\frac{4}{d}}=Q,
\]
which supply a stationary solution to (\ref{mass criticla NLS}) with
$u(t,x)=e^{it}Q(x)$.},
\begin{equation}
\begin{cases}
i\partial_{t}u+\triangle u=-|u|^{\frac{4}{d}}u. & (t,x)\in\mathbb{R}\times\mathbb{R}^{d},\\
u\mid_{t=0}=u_{0}.
\end{cases}\label{mass criticla NLS}
\end{equation}
 And Dodson improved this result in $L^{2}$ setting in \cite{Dodson 2021-1,Dodson 2021-2}.
The parallel result in Hartree-equation setting can be found in \cite{key-16}
and \cite{key-17}. 

When the mass is below the ground state $Q(x)$, Weinstein proved
the solution is global and scattering in $H^{1}$ in \cite{Weinstein 1982}
and Dodson improved this result to $L^{2}$ setting in \cite{Dodson 2015}.

When the mass is beyond the ground state, Bourgain and Wang in \cite{Bourgain-W}
constructed a type of blow up solutions with the blow up rate $||\nabla u(t)||_{L^{2}}\sim\frac{1}{T-t}$.
Besides, another type of blow up solutions, the log-log blow up solutions,
are suggested numerically by Landman, Papanocolaou, Sulem, Sulem in
\cite{numical check 1988}. This kind of solutions blow up in finite
time with the rate $||\nabla u(t)||_{L^{2}}\sim(\frac{\log|\log(T-t)|}{T-t})^{\frac{1}{2}}$
. Perelman firstly constructed this kind of solutions in her work
\cite{key-18}. After that, the log-log blow-up solutions have been
studied in depth and comprehensively by Merle and Rapha\"el in a
series of work \cite{key-8,key-9,key-10,key-19,key-20,key-21}. They
give a more complete portrayal of the blow up rate in the vicinity
of the ground state solution and a classification of the blow up solutions.

We should remark, for the mass-critical focusing NLS, when the mass
is much larger than threshold, Merle constructed a $k$-points blow
up solution in \cite{Merle 1990}, and readers can also refer to \cite{F.P P.R 2007,Fan Chenjie 2017}
for the study of weakly interacting multi bubbles blow up dynamics
for NLS. Besides, Martel and Rapha\"el in \cite{Y.M P.R 2018} constructed
a multi-bubbles blow up solution, with the rate $||\nabla u(t)||_{L^{2}}\sim\frac{|\log(T-t)|}{T-t}$
due to strong interactions.
\begin{itemize}
\item mass-supercritical and energy-subcritical case
\end{itemize}
Between this range, for the focusing 3d cubic NLS, Merle and Rapha\"el
in \cite{Main reference of FM and PR} gave a universal blow-up lower
bound in the radial case. Besides, towards this model, there is a
series of work concerning the following quantity,\footnote{$M(u):=\int|u|^{2}$ and $E(u):=\frac{1}{2}\int|\nabla u|^{2}-\frac{1}{4}\int|u|^{4}$.
And $Q$ is the unique $H^{1}$ radial positive solution of 
\[
\triangle Q-\frac{1}{2}Q+|Q|^{2}Q=0.
\]
}
\[
M\varXi:=\frac{M[u]E[u]}{M[Q]E[Q]},
\]

readers can refer to for \cite{Roudenko 2008 CMP,Roudenko 2008 RMI,Roudenko 2010 CPDE}
when $M\varXi\in(0,1)$, \cite{S. Roudenko 2008 =00003D1} for $M\varXi=1$
and \cite{K. Nakanishi and W. Schlag 2012 >1,Duyckaerts and Roudenko 2015 CMP}
in the case $M\varXi>1$.

For the focusing Hartree equation, readers may also refer to \cite{key-10-1}
in this range. Our result is also a step toward understanding some
universal blow up behaviour and the connection between NLS and Hartree
type equations.
\begin{itemize}
\item energy-critical case 
\end{itemize}
For the focusing energy-critical Schr\"odinger equation, Kenig and
Merle in \cite{S. Roudenko 2008 =00003D1} used a concentration compactness
argument and a rigidity theorem to prove any radial solutions $u(t)$
in $d=3,4,5$ which satisfy $E(u_{0})<E(W)$ and $||u_{0}||_{\dot{H}^{1}}<||W||_{\dot{H}^{1}}$\footnote{The $W(x)$ is the unique radial positive solution in $\dot{H}^{1}$
to 
\[
\triangle W+|W|^{\frac{4}{d-2}}W=0,
\]

with $d\geq3$.} must be global and scatter. For the nonradial case, readers can refer
to \cite{Dodson 2019} in $d=4$ and \cite{key-30} $d\geq5$.

Returning to the Hartree-equation setting, \cite{key-35} treated
the focusing case, and proved a parallel result with \cite{key-30}. 

\subsection{A review of Merle and Raphael's work in \cite{Main reference of FM and PR}
and the connection with our result}

Since our work rely on the method developed by Merle and Rapha\"el
in \cite{Main reference of FM and PR}, let us review some points
and highlight the main quantities in their analysis. There is also
some interesting work related to this topic, one can refer to \cite{Tao 2019}
for Navier Stokes equation, \cite{key-12} for focusing nonlinear
Klein-Gordon equation and \cite{Y.M P.R 2018} for inhomogeneous nonlinear
Schr\"odinger equation.

In this subsection, in order to make the idea of the article more
concise and clear, we only review a weaker version of Merle and Rapha\"el.
More precisely, we impose the following non-positive energy assumption
for Theorem \ref{thm PR and FM},

\label{non-positive energy}
\[
E(u_{0})\leq0,
\]
and the critical norm under consideration becomes $\dot{H}^{\frac{1}{2}}$.
Towards the nonlinear Schr\"odinger equation (\ref{eq:Schrodinger}),
for $t$ close enough to the blow up time $T$, we renormalize $u(t)$
with its $\dot{H}^{1}$ norm, i.e. let 
\begin{equation}
\lambda_{u}(t)=(\frac{1}{||\nabla u(t)||_{L^{2}}})^{2},\label{eq:scaling-4}
\end{equation}
 and define the renormalization of $u(t)$:
\begin{equation}
v(\tau,x)=\lambda_{u}(t)\bar{u}(t-\lambda_{u}(t)^{2}\tau,\lambda_{u}(t)x).\label{eq:renorm-1}
\end{equation}

From local theory of NLS, we note that 
\begin{equation}
\lambda_{u}(t)\lesssim\sqrt{T-t}.\label{the scaling lower bound}
\end{equation}

Then a long time behaviour of blow up dynamics has been transformed
into a Cauchy problem that satisfies the following special conditions:
\begin{equation}
\begin{cases}
i\partial_{\tau}v+\triangle v(\tau,x)=|v|^{2}v(\tau,x), & (\tau,x)\in[0,\frac{1}{\lambda_{u}(t)}]\times\mathbb{R}^{3},\\
v(\tau,x)\mid_{\tau=0}=\lambda_{u}(t)\bar{u}(t,\lambda_{u}(t)x),
\end{cases}\label{special Cauchy prob}
\end{equation}

with 
\begin{equation}
||v(0,x)||_{\dot{H}^{\frac{1}{2}}}=||u(t)||_{\dot{H}^{\frac{1}{2}}},\label{the scaling norm}
\end{equation}
\begin{equation}
||v(0,x)||_{\dot{H}^{1}}=1,\label{the scaling norm of H1}
\end{equation}

and the non-positive energy
\begin{equation}
E(v_{0})\leq0.\label{non-positive energy for v}
\end{equation}

Let

\begin{equation}
N(t)=-\log\lambda_{u}(t),\label{time transformation}
\end{equation}
we could conclude, Merle and Rapha\"el converted Theorem \ref{thm PR and FM}
into proving the following conclusion, there exists a universal constant
$\gamma>0$ such that 
\begin{equation}
||v(0,x)||_{\dot{H}^{\frac{1}{2}}}\geq N^{\gamma}.\label{the goal}
\end{equation}

The above analysis reduces the difficulty of the problem to some extent.
We know that the blow up behaviour of the solution is a long-time dynamic
behaviour, and there are few tools that can be directly applied to
characterize the explosion. However, if we are now looking at the
local behaviour of the solution, then a rich local theory of the solutions
can be applied, making the problem possible. 

From the above analysis, it is also a very natural thing to give the
connection between the lower bounds of the blow up rate of the Schr\"{o}dinger
equation and Hartree equation, which also gives the possibility to
apply the local theory. 

To further explain the work by Merle and Rapha\"el, we now introduce
the following notations, the scaling invariant Morrey-Campanato norm
\begin{equation}
\rho(u,R)=\sup_{R'\geq R}\frac{1}{R'}\int_{R'\leq|x|\leq2R'}|u(x)|dx,\label{Morrey-Campanato norm}
\end{equation}

a quantity related to the initial data,
\[
M_{0}=\frac{4||v(0)||_{\dot{H}^{\frac{1}{2}}}}{C_{GN}},
\]

where $C_{GN}$ is a universal constant related to Gagliardo-Nirenberg
inequality and a similar definition to (\ref{eq:scaling-4}),
\begin{equation}
\lambda_{v}(\tau)=\frac{1}{||v(\tau)||_{\dot{H}^{1}}^{2}}.\label{scaling-5-1}
\end{equation}
.

In order to achieve the goal (\ref{the goal}), there are three steps
need to be implemented,

Step 1. Uniform control of the $\rho$ norm and the dispersive estimate.

To achieve the first goal, we give the following proposition,
\begin{prop}
\label{prop: p norm and H1 norm} Let $v(\tau)\in C([0,e^{N}],\dot{H}^{\frac{1}{2}}\cap\dot{H}^{1})$
be a radially symmetric solution to (\ref{eq:Schrodinger}) where
$N$ is a sufficiently large number and (\ref{non-positive energy for v}),
(\ref{the scaling norm of H1}) hold. Then there exist universal constants
$C_{1}$, $\alpha_{1}$ and $\alpha_{2}$ such that the following
hold, $\forall\tau_{0}\in[0,e^{N}]$,

(i) the uniform control of the $\rho$ norm:
\begin{equation}
\rho(v(\tau_{0}),C_{1}M_{0}^{\alpha_{1}}\sqrt{\tau_{0}})\leq C_{1}M_{0}^{2}.\label{the control of p norm}
\end{equation}

(ii) the dispersive estimate:
\begin{equation}
\int_{0}^{\tau_{0}}(\tau_{0}-\tau)||v(\tau)||_{\dot{H}^{1}}^{2}d\tau\leq C_{1}M_{0}^{\alpha_{2}}\tau_{0}^{\frac{3}{2}}.\label{the dispersive estimate}
\end{equation}

Moreover, if we assume $M_{0}^{\alpha_{2}}<e^{\frac{\sqrt{N}}{2}}$,
there exist a sequence of $\{\tau_{i}\}$ with $\tau_{i}\in[0,e^{i}]$
and $i\in\{\sqrt{N},\sqrt{N}+1,...,N\}$ with the following bounds
holds, 
\begin{equation}
\frac{\sqrt{\tau_{i}}}{\lambda_{v}(\tau_{i})}\leq C_{1}M_{0}^{\alpha_{2}}\;\;\and\;\;\lambda_{v}(\tau_{i})\in[\frac{1}{10C_{1}M_{0}^{\alpha_{2}}}e^{\frac{i-1}{2}},\frac{10}{C_{1}M_{0}^{\alpha_{2}}}e^{\frac{i}{2}}].\label{the chosen time}
\end{equation}
\end{prop}
There are two points we should note,
\begin{rem}
$\begin{cases}
\alpha_{1}=1,\\
\alpha_{2}=5
\end{cases}$ is an allowable value of the above proposition, and it is related
to the blow up rate in Theorem \ref{thm PR and FM}, we could see
it more clearly to achieve step 2.
\end{rem}
\begin{rem}
The conclusion (\ref{the chosen time}) in proposition \ref{prop: p norm and H1 norm}
is a direct inference of (\ref{the dispersive estimate}) with the
help of (\ref{the scaling norm of H1}). Since (\ref{the dispersive estimate})
only gives an estimate of the decay of $||v(\tau)||_{\dot{H}^{1}}$
in the average sense, we can only get (\ref{the chosen time}) at
some special time rather than a pointwise estimate. 
\end{rem}
Step 2. Lower bound on a weighted local $L^{2}$ norm of $v(0)$.

Now we are ready to give a uniform lower bound on a weighted local
$L^{2}$ norm of $v(0)$, 
\begin{prop}
Let $v(\tau)$ satisfies the conditions in proposition (\ref{prop: p norm and H1 norm})
and $\{\tau_{i}\}$ are chosen in proposition (\ref{prop: p norm and H1 norm}),
then the following weighted $L^{2}$ norm on a sufficiently large
ball holds,
\begin{equation}
\frac{1}{\lambda_{v}(\tau_{i})}\int_{|x|\leq M_{0}^{2+2\alpha_{2}}\lambda_{v}(\tau_{i})}|v(0)|^{2}\geq c_{3},\label{the lower bound ona ball}
\end{equation}

where $c_{3}$ is a universal constant. 
\end{prop}
\begin{rem}
The idea of proof is the following: first we derive the lower bound
at time $\tau=\tau_{i}$ with the aid of non-positive energy constraint,
i.e.
\begin{equation}
\frac{1}{\lambda_{v}(\tau_{i})}\int_{|x|\leq M_{0}^{2+2\alpha_{2}}\lambda_{v}(\tau_{i})}|v(\tau_{i})|^{2}\geq2c_{3}.\label{eq:lower bound-2}
\end{equation}
\end{rem}
Second, thanks to (\ref{the control of p norm}) and (\ref{the dispersive estimate}),
we conclude the difference between $\tau=0$ and $\tau=\tau_{i}$
is sufficiently small, i.e. $\exists\;\varepsilon>0$ small enough,
such that 
\begin{equation}
|\frac{1}{\lambda(\tau_{i})}\int_{|x|\leq M_{0}^{2+2\alpha_{2}}\lambda_{v}(\tau_{i})}|v(\tau_{i})|^{2}-\frac{1}{\lambda(\tau_{i})}\int_{|x|\leq M_{0}^{2+2\alpha_{2}}\lambda_{v}(\tau_{i})}|v(0)|^{2}|<\varepsilon.\label{eq:lower bound-3}
\end{equation}

Combining (\ref{eq:lower bound-2}) and (\ref{eq:lower bound-3}),
we derive the final conclusion (\ref{the lower bound ona ball}). 
\begin{rem}
The lower bound (\ref{the lower bound ona ball}) has implied the
choice of $\gamma$ in Theorem \ref{thm PR and FM}. A straightforward
algebraic computation of our analysis implies
\[
\gamma=\frac{1}{4+2\alpha_{2}+\varepsilon},
\]

for $\forall\varepsilon>0$, i.e we could derive $\gamma=\frac{1}{14+\varepsilon}$
in this setting while Merle and Rapha\"el in \cite{Main reference of FM and PR}
give 

\begin{equation}
\gamma=\frac{1}{12}.\label{blow up rate of PR}
\end{equation}
\end{rem}
Step 3. The construction of $N^{\gamma}$ disjoint annuli.

Now we are going to construct sufficiently many disjoint annuli. We
need to pay attention to the following two points,
\begin{itemize}
\item These $N^{\gamma}$ annuli are disjoint.
\item At each annulus, the similar lower bound still holds in (\ref{the lower bound ona ball}).
\end{itemize}
A direct computation by H\"older inequality implies the following
choice of the size of annuli satisfies the second point,
\[
\mathscr{C}_{i}:=\{x\in\mathbb{R}^{3}\mid\frac{\lambda_{v}(\tau_{i})}{M_{0}^{2+\alpha_{2}}}\leq|x|\leq\lambda_{v}(\tau_{i})M_{0}^{2+\alpha_{2}}\}.
\]

In order to satisfy the first point above, we need to choose the number
$p$ such that 
\[
\frac{\lambda_{v}(\tau_{i+p})}{M_{0}^{2+\alpha_{2}}}\geq\lambda_{v}(\tau_{i})M_{0}^{2+\alpha_{2}}.
\]

Thanks to (\ref{the chosen time}), we can choose $p$ such that $e^{\frac{p}{2}}\thickapprox M_{0}^{4+2\alpha_{2}}$
which gives the derived choice of $N^{\gamma}$ disjoint annuli.

Based on the above three steps, we have the following technical remark,
\begin{rem}
\label{a weak version}The condition (\ref{non-positive energy for v})
can be weaken by another three assumptions, and the stronger version
will be stated in section 3. Here we are only using condition (\ref{non-positive energy for v})
to state this proposition for convenience.
\end{rem}
Now we want to talk about the relation in this paper between NLS and
Hartree equation,
\begin{itemize}
\item Using the above techniques for proving Theorem \ref{thm PR and FM},
we can obtain result that is parallel in the Hartree equation, which
is also the purpose of our article.
\item If we get the conclusion in Hartree equation directly, then we can
approximate $\delta$ function by doing a suitable scaling transformation
of the potential function $V(x)$\footnote{We should note that the conditions (\ref{connection condition}),
(\ref{the integrability condition}), (\ref{eq:technisch reason})
is invariant under the transformation (\ref{a scaling transformation for potential function}).},
\begin{equation}
V_{\varepsilon}(x)=\frac{1}{\varepsilon^{3}}V(\frac{x}{\varepsilon}),\;\forall\varepsilon>0.\label{a scaling transformation for potential function}
\end{equation}
 After using the mature local theory\footnote{The above three steps are local versions for the dynamics of the solutions,
so we could apply local theory.}, we prove that the Schr\"odinger equation is also valid in this
case.
\end{itemize}
The above two points give a connection between NLS and Hartree equation,
and we supply complete proof in section 3.

\subsection{Strategy and structure of the paper}

We use the robust strategy by Merle and Rapha\"el in \cite{Main reference of FM and PR}.
We should note that there are two key points:

(1) The additional error terms.

The condition (\ref{the integrability condition}) is natural assumption
in our problem setting, and the role of the condition (\ref{connection condition})
is to connect the energy (\ref{eq:energy-1}) with the virial identity
(\ref{virial identity in Hartree seting}). Another assumption (\ref{eq:technisch reason})
is technical, because we need to control additional error terms which
do not exist in the Schr\"odinger-equation setting. As in the Schr\"odinger
equation, the ``potential function'' is $\delta(x)$. However, in
our setting, $V(x)$ belongs to $L^{1}(\mathbb{R}^{3})$. So, when
the blow-up phenomenon happens, there are some additional error terms
we should treat. We give a priori control on two channels which helps
us to overcome this difficulty. For more details, one can see step
2 in Proposition \ref{prop 1 of ours}. 

(2) The difficulty caused by the no-scaling property of $V(x)$. 

Since we treat the no-scaling case for the potential function $V(x)$,
we should check carefully the constants chosen in Proposition \ref{prop 1 of ours}
and \ref{prop 2 of ours}. After the renormalization of $u(t)$, $v^{\lambda(t)}(\tau,x):=\lambda_{u}(t)\bar{u}(t-\lambda_{u}^{2}(t)\tau,\lambda_{u}(t)x)$
at different time $t$ satisfy different equations 
\begin{equation}
\begin{cases}
i\partial_{\tau}v^{\lambda(t)}+\triangle v^{\lambda(t)}=-(V_{\lambda(t)}*|v^{\lambda(t)}|^{2})v^{\lambda(t)}. & (\tau,x)\in[0,\frac{t}{\lambda(t)^{2}})\times\mathbb{R}^{3},\;V_{\lambda(t)}(x)=\lambda(t)^{3}V(\lambda(t)x),\\
v^{\lambda(t)}\mid_{\tau=0}=\lambda(t)\bar{u}(t,\lambda(t)x)\in\dot{H}^{\frac{1}{2}}\cap\dot{H}^{1},
\end{cases}\label{eq:scaling equation}
\end{equation}
where $\lambda_{u}(t)=(\frac{1}{||\nabla u(t)||_{L^{2}}})^{2}$. For
these $v^{\lambda(t)}(\tau,x)$, we should give a uniform estimate
which is independent of the time $t$. We will deal with this very
carefully both in Proposition \ref{prop 1 of ours} and \ref{prop 2 of ours},
which is also the core of our analysis. Here, we explain from two
aspects why the potential function does not have scaling invariant
property will not have an essential impact for our analysis.

(i) The renormalization of $u(t)$

$v^{\lambda(t)}(\tau,x)$ defined above satisfies equation (\ref{eq:scaling equation}),
and it also follows the law of conservation of energy:
\begin{align}
E^{\lambda}(v^{\lambda(t)}(\tau)) & =E^{\lambda}(v^{\lambda(t)}(0))=\frac{1}{2}\int|\nabla v^{\lambda(t)}(0,x)|^{2}dx-\frac{1}{4}\int(V_{\lambda(t)}*|v^{\lambda(t)}|^{2})(0,x)|v^{\lambda(t)}|^{2}(0,x)dx\label{eq:scaling energy}\\
 & =\lambda(t)[\frac{1}{2}\int|\nabla u(t,x)|^{2}dx-\frac{1}{4}\int(V*|u|^{2})|u|^{2}(t,x)dx]\nonumber \\
 & =\lambda(t)E(u(t)).\nonumber 
\end{align}

(\ref{eq:scaling energy}) implies we still obtain the scaling invariant
property for the energy conservation law although $V(x)$ does not
own it. 

(ii) The scaling for the potential function $V(x)$.

A direct computation implies 
\begin{equation}
|\sum x_{j}\partial_{x_{j}}V_{\lambda}(x)|=\lambda^{3}|\sum\lambda x_{j}\partial_{\lambda x_{j}}V(\lambda x)|\leq\lambda^{3}\cdotp\frac{C}{|\lambda x|^{3}}=\frac{C}{|x|^{3}},\label{eq:scaling potential}
\end{equation}

under the assumption (\ref{eq:technisch reason}) which is a universal
upper bound independent of $\lambda$. Besides, we only use $L^{1}$
norm for the potential function $V(x)$ which is also a scaling invariant
norm under the renormalization, i.e. $||V_{\lambda}(x)||_{L^{1}}=||V(x)||_{L^{1}}$. 

The paper is organized as follows. In section 2, we will prove Theorem
\ref{thm 2 of ours}. More specifically, we will give a uniform control
of the $\rho$ norm, which is a suitable scaling invariant Morrey-Campanato
norm, and a lower bound on a weight local $L^{2}$ norm, then we use
a blackbox by applying Proposition \ref{prop 1 of ours} and \ref{prop 2 of ours}
to finish the proof for Theorem \ref{thm 2 of ours}. In section 3,
we supply the connection between Schr\"odinger equation and Hartree
equation. In Appendix A, for readers' convenience, we give some examples
of blow up solutions and some direct observations for the blow up
rate for $V(x)$ with higher integrability conditions. . However,
we can not give any special examples $V(x)$ which lead to this kind
of blow up solutions. Appendix B is devoted to give some standard
result of stability theory needed in Section 3. 

\section{BLOW-UP RATE FOR THE BLOW-UP SOLUTION}

In this section, we focus on the proof of Theorem \ref{thm 2 of ours},
i.e. we consider the equation
\begin{equation}
\begin{cases}
i\partial_{t}u+\triangle u=-(V*|u|^{2})u. & (t,x)\in\mathbb{R}\times\mathbb{R}^{3},\\
u\mid_{t=0}=u_{0}\in\dot{H}^{\frac{1}{2}}\cap\dot{H}^{1}.
\end{cases}\label{eq:Hartree-4}
\end{equation}

with $V(x)$ satisfies the conditions in Theorem \ref{thm 2 of ours}.
We will give a lower bound for the blow-up rate.

First, we prove the main propositions at the heart of the proof of
Theorem \ref{thm 2 of ours}. Let $u_{0}\in\dot{H}^{\frac{1}{2}}\cap\dot{H}^{1}$
with radial symmetry and assume that the corresponding solution $u(t)$
to (\ref{eq:Hartree-4}) blows up in finite time $0<T<+\infty$. We
can pick $t$ close enough to $T$. Let 
\begin{equation}
\lambda_{u}(t)=(\frac{1}{||\nabla u(t)||_{L^{2}}})^{2},\label{eq:scaling-2}
\end{equation}

then from the local theory,
\begin{equation}
\lambda_{u}(t)\lesssim\sqrt{T-t}.\label{eq:scaling-3}
\end{equation}

We define the renormalization of $u(t)$ by 
\begin{equation}
v^{(t)}(\tau,x)=\lambda_{u}(t)\bar{u}(t-\lambda_{u}^{2}(t)\tau,\lambda_{u}(t)x).\label{eq:renorm-2}
\end{equation}

In this subsection, to clarify the notations, we omit the dependence
of $\lambda$ on $u(t)$ and define $v^{\lambda}:=v^{(t)}(\tau,x)$.

Then it is not hard to check that $v^{\lambda}$ satisfies the following
equation 
\begin{equation}
\begin{cases}
i\partial_{\tau}v^{\lambda}+\triangle v^{\lambda}=-(V_{\lambda}*|v^{\lambda}|^{2})v^{\lambda}. & (\tau,x)\in[0,\frac{t}{\lambda^{2}})\times\mathbb{R}^{3},\;V_{\lambda}(x)=\lambda^{3}V(\lambda x),\\
v^{\lambda}\mid_{\tau=0}=\lambda\bar{u}(t,\lambda x)\in\dot{H}^{\frac{1}{2}}\cap\dot{H}^{1}.
\end{cases}\label{eq:scaling-3-1}
\end{equation}

Because of the no-scaling property of the potential function $V(x)$,
the renormalization $v^{\lambda}$ satisfies different equation at
different time $t$. So we need some universal properties for these
$v^{\lambda}$ which are independent of $\lambda$. 

We first give needed definitions in our following propositions and
recall some elementary inequalities.

Define a smooth radially symmetric cut-off function,
\begin{equation}
\psi(x)=\begin{cases}
\frac{|x|^{2}}{2}, & |x|\leq2,\\
0, & |x|\geq3,
\end{cases}\label{cut-off function}
\end{equation}

with $\psi_{R}(x):=R^{2}\psi(\frac{x}{R})$ and we state the following
lemma in \cite{Main reference of FM and PR}:
\begin{lem}
\label{sobolev lemma 1}(Radial Gagliardo-Nirenberg inequality)

(i) There exists a universal constant $C>0$ such that for all $u\in L^{3}$,
\begin{equation}
\forall R>0,\;\frac{1}{R}\int_{|y|\leq R}|u|^{2}dy\leq C||u||_{L^{3}}^{2}.\label{eq:sobolev-1}
\end{equation}

(ii) For all $\eta>0,$ there exists a constant $C_{\eta}>0$ such
that for all $u\in\dot{H}^{\frac{1}{2}}\cap\dot{H}^{1}$ with radial
symmetry, for all $R>0,$
\begin{equation}
\int_{|x|\geq R}|u|^{4}\leq\eta||\nabla u||_{L^{2}(|x|\geq R)}^{2}+\frac{C_{\eta}}{R}[\rho(u,R)^{2}+\rho(u,R)^{3}].\label{eq:sobolev-2}
\end{equation}
\end{lem}
And from the definition of energy, we know 
\begin{align}
E^{\lambda}(v^{\lambda}) & :=\frac{1}{2}\int|\nabla v^{\lambda}|^{2}dx-\frac{1}{4}\int(V_{\lambda}*|v^{\lambda}|^{2})|v^{\lambda}|^{2}dx\label{eq:sobolev-3}\\
 & \geq\frac{1}{2}(1-(\frac{||v^{\lambda}||_{L^{3}}}{c_{V}})^{2})\int|\nabla v^{\lambda}|^{2}dx,\nonumber 
\end{align}

for some $c_{V}>0$ which is independent of $\lambda$. 

With these prepared knowledge, we prove the following,
\begin{prop}
\label{prop 1 of ours}Let $\tau_{*}>0$ and $v^{\lambda}(\tau)\in C([0,\tau_{*}],\dot{H}^{\frac{1}{2}}\cap\dot{H}^{1})$
be a radially symmetric solution to (\ref{eq:scaling-3-1}) and assume
\begin{equation}
\tau_{*}^{\frac{1}{2}}max(E^{\lambda}(v_{0}^{\lambda}),0)<1,\label{scaling energy-3}
\end{equation}

and 
\begin{equation}
M_{0}^{\lambda}:=\frac{4||v_{0}^{\lambda}||_{L^{3}}}{c_{V}}\geq2,\label{scaling initial data-3}
\end{equation}

then there exist $C_{1},\alpha_{1},\alpha_{2}>0$ which are independent
of $\lambda$ such that 
\begin{equation}
\rho(v^{\lambda}(\tau_{*}),(M_{0}^{\lambda})^{\alpha_{1}}\sqrt{\tau_{*}})\leq C_{1}(M_{0}^{\lambda})^{2},\label{upper M-C norm}
\end{equation}

and 
\begin{equation}
\int_{0}^{\tau_{*}}(\tau_{*}-\tau)||\nabla v^{\lambda}(\tau)||_{L_{x}^{2}}^{2}d\tau\leq(M_{0}^{\lambda})^{\alpha_{2}}\tau_{*}^{\frac{3}{2}}.\label{upper H1 norm}
\end{equation}
\end{prop}
\begin{proof}
Step 1. Local radial virial estimate.

We define 
\begin{equation}
V_{a}(\tau):=\int_{R^{3}}a(x)|v^{\lambda}(\tau,x)|^{2}dx,\label{weight L2}
\end{equation}

and 
\begin{equation}
P_{a}(\tau):=\frac{d}{d\tau}V_{a}(\tau)=2Im\int_{R^{3}}\nabla a(x)\nabla v^{\lambda}(\tau,x)\bar{v^{\lambda}}(\tau,x)dx.\label{virial identity-1}
\end{equation}

The a direct computation gives 
\begin{align*}
\frac{1}{4}\frac{d}{d\tau}P_{a}(\tau) & \leq\alpha E^{\lambda}(v^{\lambda}(\tau))-(\frac{\alpha}{2}-1)\int|\nabla v^{\lambda}|^{2}dx+\frac{C}{R^{2}}\int_{2R\leq|x|\leq3R}|v^{\lambda}|^{2}dx\\
 & +\frac{1}{4}\int\int[(\partial_{x_{j}}\psi_{R}(x)-\partial_{y_{j}}\psi_{R}(y))-(x_{j}-y_{j})]\partial_{x_{j}}V_{\lambda}(x-y)|v^{\lambda}(\tau,y)|^{2}|v^{\lambda}(\tau,x)|^{2}dydx\\
 & =\alpha E^{\lambda}(v^{\lambda}(\tau))-(\frac{\alpha}{2}-1)\int|\nabla v^{\lambda}|^{2}dx+\frac{C}{R^{2}}\int_{2R\leq|x|\leq3R}|v^{\lambda}|^{2}dx+\Phi(\tau),
\end{align*}

where we substitute $a(x)$ with $\psi_{R}(x)$ and use the condition
(\ref{connection condition}) with 
\begin{align*}
|\Phi(\tau)| & =\frac{1}{4}|\int\int[(\partial_{x_{j}}\psi_{R}(x)-\partial_{y_{j}}\psi_{R}(y))-(x_{j}-y_{j})]\partial_{x_{j}}V_{\lambda}(x-y)|v^{\lambda}(\tau,y)|^{2}|v^{\lambda}(\tau,x)|^{2}dydx|\\
 & \lesssim\iint_{|x|\geq R,|y|\geq R}|[(\partial_{x_{j}}\psi_{R}(x)-\partial_{y_{j}}\psi_{R}(y))-(x_{j}-y_{j})]\partial_{x_{j}}V_{\lambda}(x-y)||v^{\lambda}(\tau,y)|^{2}|v^{\lambda}(\tau,x)|^{2}dxdy\\
 & +\iint_{|x|\leq R,|y|\geq2R}|[(\partial_{x_{j}}\psi_{R}(x)-\partial_{y_{j}}\psi_{R}(y))-(x_{j}-y_{j})]\partial_{x_{j}}V_{\lambda}(x-y)||v^{\lambda}(\tau,y)|^{2}|v^{\lambda}(\tau,x)|^{2}dxdy\\
 & +\iint_{|y|\leq R,|x|\geq2R}|[(\partial_{x_{j}}\psi_{R}(x)-\partial_{y_{j}}\psi_{R}(y))-(x_{j}-y_{j})]\partial_{x_{j}}V_{\lambda}(x-y)||v^{\lambda}(\tau,y)|^{2}|v^{\lambda}(\tau,x)|^{2}dxdy\\
 & :=(I)+(II)+(III).
\end{align*}

We should point out,\textcolor{red}{{} }the term $\Phi(\tau)$ is different
from the error term in Schr\"odinger-equation setting and this is
the reason to assume the additional condition (\ref{eq:technisch reason}).
We summarize the estimate of $\Phi(\tau)$ as follows, 
\begin{lem}
\label{sobolev lemma 2}Assume (\ref{the integrability condition}),
(\ref{eq:technisch reason}) in Theorem \ref{thm 2 of ours} hold,
then $\forall\eta>0,\;0<R_{1}<R,$ there exists $C(\eta)>0$ such
that 
\begin{align}
|\Phi(\tau)| & \leq\eta||\nabla v^{\lambda}(\tau)||_{L_{x}^{2}}^{2}+C(\eta)R^{-1}[\rho(v^{\lambda}(\tau),R)+\rho(v^{\lambda}(\tau),R)^{3}+\rho(v^{\lambda}(\tau),R_{1})\rho(v^{\lambda}(\tau),R)]\label{additional error term}\\
 & +||\nabla v^{\lambda}||_{L^{2}}^{2}\cdotp\frac{R_{1}^{2}}{R^{2}}\cdotp\rho(v^{\lambda}(\tau),R).\nonumber 
\end{align}
\end{lem}
If we assume Lemma \ref{sobolev lemma 2} holds, then we conclude
\begin{align}
\frac{1}{4}\frac{d}{dt}P_{a}(\tau) & \leq\alpha E^{\lambda}(v^{\lambda}(\tau))-\frac{1}{2}(\frac{\alpha}{2}-1)\int|\nabla v^{\lambda}|^{2}dx\label{local virial identity}\\
 & +C(\eta)R^{-1}[\rho(v^{\lambda}(\tau),R)+\rho(v^{\lambda}(\tau),R)^{3}+\rho(v^{\lambda}(\tau),R_{1})\rho(v^{\lambda}(\tau),R)]+||\nabla v^{\lambda}||_{L^{2}}^{2}\cdotp\frac{R_{1}^{2}}{R^{2}}\cdotp\rho(v^{\lambda}(\tau),R).\nonumber 
\end{align}

Now we return to the proof of Lemma \ref{sobolev lemma 2}, 
\begin{proof}
For $(I)$, 
\begin{align}
(I) & \leq||x\cdotp\nabla_{x}V(\lambda x)||_{L^{1}}||v^{\lambda}(\tau,x)\cdotp\chi_{\{|x|\geq R\}}||_{L^{4}}^{2}||v^{\lambda}(\tau,y)\cdotp\chi_{\{|y|\geq R\}}||_{L^{4}}^{2}\label{term I}\\
 & \leq||v^{\lambda}(\tau,x)\cdotp\chi_{\{|x|\geq R\}}||_{L^{4}}^{4}\nonumber \\
 & \lesssim\eta||\nabla v^{\lambda}||_{L^{2}}^{2}+\frac{C_{\eta}}{R}[\rho(v^{\lambda},R)^{2}+\rho(v^{\lambda},R)^{3}]\nonumber 
\end{align}

by Lemma \ref{sobolev lemma 1}.

Since $(II)$ and $(III)$ are symmetric, we only estimate $(II)$.
In fact, this term does not appear in the Schr\"odinger equation.
We should divide the area $\{|x|\leq R,|y|\geq2R\}$ into $\{|x|\leq R_{1},|y|\geq2R\}$
and $\{R_{1}\leq|x|\leq R,|y|\geq2R\}$ and choose $R_{1}$ carefully
in next step. The term in the area $\{|x|\leq R_{1},|y|\geq2R\}$
can be controlled by $\frac{R_{1}^{2}}{R^{2}}||\nabla v^{\lambda}||_{L^{2}}^{2}\rho(v^{\lambda}(\tau),R)$
and the other is controlled by $\frac{1}{R}\rho(v^{\lambda}(\tau),R_{1})\rho(v^{\lambda}(\tau),R)$.
The details are the followings.

In the area $\{|x|\leq R,|y|\geq2R\},$ using the property (\ref{eq:technisch reason}),
we conclude 
\begin{equation}
|[(\partial_{x_{j}}\psi_{R}(x)-\partial_{y_{j}}\psi_{R}(y))-(x_{j}-y_{j})]\partial_{x_{j}}V_{\lambda}(x-y)|\lesssim\lambda^{3}\cdotp\frac{1}{(\lambda|x-y|)^{3}}=\frac{1}{|x-y|^{3}}\lesssim\frac{1}{|y|^{3}},
\end{equation}

which leads to 
\begin{align}
(II) & \lesssim|\iint_{|x|\leq R,|y|\geq2R}\frac{1}{|y|^{3}}|v^{\lambda}(\tau,y)|^{2}|v^{\lambda}(\tau,x)|^{2}dxdy|\label{term II}\\
 & \lesssim\int_{|x|\leq R}|v^{\lambda}(\tau,x)|^{2}dx[\sum_{j=0}^{+\infty}\frac{1}{(2^{j}R)^{3}}\int_{2^{j}R\leq|y|\leq2^{j+1}R}|v^{\lambda}(\tau,y)|^{2}dy]\nonumber \\
 & \lesssim[\int_{|x|\leq R_{1}}|v^{\lambda}(\tau,x)|^{2}dx+\int_{R_{1}\leq|x|\leq R}|v^{\lambda}(\tau,x)|^{2}dx]\cdotp\frac{1}{R^{2}}\rho(v^{\lambda}(\tau),R)\nonumber \\
 & \lesssim[R_{1}^{2}||\nabla v^{\lambda}||_{L^{2}}^{2}+R\rho(v^{\lambda}(\tau),R_{1})]\cdotp\frac{1}{R^{2}}\cdotp\rho(v^{\lambda}(\tau),R).\nonumber 
\end{align}

From (\ref{term I}) and (\ref{term II}), we finish the proof of
Lemma \ref{sobolev lemma 2}.
\end{proof}
Step 2. A priori control of the $\rho$ norm on parabolic space time
interval. 

In this step, we need to choose $R_{1}$ appropriately. On the one
hand, we should choose $R_{1}$ small enough so that $\frac{R_{1}^{2}}{R^{2}}\cdotp\rho(v^{\lambda}(\tau),R)\ll\frac{1}{2}$.
On the other hand, we can not choose $R_{1}$ too small otherwise
we can not control $R^{-1}\rho(v^{\lambda}(\tau),R_{1})\rho(v^{\lambda}(\tau),R)$.
However, the above analysis is not self-contradictory. A natural idea
is, we choose $R_{1}$ with $\frac{R_{1}^{2}}{R^{2}}\lesssim\varepsilon^{c_{1}}$
such that $\frac{R_{1}^{2}}{R^{2}}\cdotp\rho(v^{\lambda}(\tau),R)\ll\frac{1}{2}$.
At the same time, let $R\geq R_{1}\gg(M_{0}^{\lambda})^{\frac{1}{2\varepsilon}}$
so that $R^{-1}\rho(v^{\lambda}(\tau),R_{1})\rho(v^{\lambda}(\tau),R)$
still can be treated as an error term. Once this difficulty is overcome,
the rest of the analysis is similar to the one in \cite{Main reference of FM and PR}.

Let $\varepsilon,\delta>0$ be a small enough constant to be chosen
later. Let 
\begin{equation}
G_{\varepsilon}=(M_{0}^{\lambda})^{\frac{1}{\varepsilon}},\;A_{\varepsilon_{1}}=(\frac{\varepsilon G_{\varepsilon}}{(M_{0}^{\lambda})^{2}})^{\frac{1}{3}},\;A_{\varepsilon_{2}}=(\frac{\varepsilon G_{\varepsilon}}{(M_{0}^{\lambda})^{2}})^{\frac{1}{3}}\cdotp(\frac{\delta}{M_{0}^{\lambda}}).\label{chosen parameter}
\end{equation}

Recall from Lemma \ref{sobolev lemma 1}, there exists a universal
constant $C>0$ such that 
\[
\forall R>0,\;\forall u\in L^{3},\;\rho(u,R)\leq C||u||_{L^{3}}^{2}.
\]

From the regularity of the flow $v^{\lambda}\in C([0,\tau^{*}],\dot{H}^{\frac{1}{2}}\cap\dot{H}^{1})$
and the definition of $M_{0}^{\lambda}$, we may consider the largest
time $\tau_{1}\in[0,\tau_{*}]$ such that 
\begin{equation}
\forall\tau_{0}\in[0,\tau_{1}],\;[M_{1}^{\lambda}(A_{\varepsilon_{1}},\tau_{1})]^{2}=\max_{\tau\in[0,\tau_{1}]}\rho(v^{\lambda}(\tau),A_{\varepsilon_{1}}\sqrt{\tau})\leq2\frac{(M_{0}^{\lambda})^{2}}{\varepsilon},\label{upper M-C norm-1}
\end{equation}

\begin{equation}
[M_{2}^{\lambda}(A_{\varepsilon_{2}},\tau_{1})]^{2}=\max_{\tau\in[0,\tau_{1}]}\rho(v^{\lambda}(\tau),A_{\varepsilon_{2}}\sqrt{\tau})\leq2\frac{(M_{0}^{\lambda})^{5}}{\varepsilon^{10}}\label{upper M-C norm-2}
\end{equation}

and 
\begin{equation}
\int_{0}^{\tau_{0}}(\tau_{0}-\tau)||\nabla v^{\lambda}(\tau)||_{L^{2}}^{2}d\tau\leq G_{\varepsilon}\tau_{0}^{\frac{3}{2}}.\label{upper H1 norm-1}
\end{equation}

We claim that : 
\begin{equation}
\forall\tau_{0}\in[0,\tau_{1}],\;[M_{1}^{\lambda}(A_{\varepsilon_{1}},\tau_{1})]^{2}=\max_{\tau\in[0,\tau_{1}]}\rho(v^{\lambda}(\tau),A_{\varepsilon_{1}}\sqrt{\tau})\leq\frac{(M_{0}^{\lambda})^{2}}{\varepsilon},\label{upper M-C norm-1-1}
\end{equation}

\begin{equation}
[M_{2}^{\lambda}(A_{\varepsilon_{2}},\tau_{1})]^{2}=\max_{\tau\in[0,\tau_{1}]}\rho(v^{\lambda}(\tau),A_{\varepsilon_{2}}\sqrt{\tau})\leq\frac{(M_{0}^{\lambda})^{5}}{\varepsilon^{10}},\label{upper M-C norm-2-1}
\end{equation}

and 
\begin{equation}
\int_{0}^{\tau_{0}}(\tau_{0}-\tau)||\nabla v^{\lambda}(\tau)||_{L^{2}}^{2}d\tau\leq\frac{G_{\varepsilon}}{2}\tau_{0}^{\frac{3}{2}}\label{upper H1 norm-2}
\end{equation}

provided $\varepsilon,\delta>0$ has been chosen small enough, and
(\ref{upper M-C norm}), (\ref{upper H1 norm}) follow. 

Proof of (\ref{upper M-C norm-1-1}) and (\ref{upper M-C norm-2-1}).
$\forall\tau_{0}\in[0,\tau_{1}]$, we integrate twice in time between
$0$ and $\tau_{0}$ from (\ref{local virial identity}) and get :
\begin{align}
 & \int\psi_{R}|v^{\lambda}(\tau_{0})|^{2}+C'_{1}\int_{0}^{\tau_{0}}(\tau_{0}-\tau)||\nabla v^{\lambda}(\tau)||_{L^{2}}^{2}d\tau\label{local virial identity-1}\\
 & \lesssim\int\psi_{R}|v^{\lambda}(0)|^{2}+\tau_{0}[Im(\int\nabla\psi_{R}(x)\nabla v^{\lambda}(0)\bar{v^{\lambda}}(0)dx)+E(v_{0}^{\lambda})\tau_{0}]\nonumber \\
 & +\frac{1}{R}\int_{0}^{\tau_{0}}(\tau_{0}-\tau)[\rho(v^{\lambda}(\tau),R)+\rho(v^{\lambda}(\tau),R)^{3}+\rho(v^{\lambda}(\tau),R_{1})\rho(v^{\lambda}(\tau),R)]d\tau\nonumber \\
 & +\frac{R_{1}^{2}}{R^{2}}\int_{0}^{\tau_{0}}(\tau_{0}-\tau)||\nabla v^{\lambda}||_{L^{2}}^{2}\rho(v^{\lambda}(\tau),R)d\tau.\nonumber 
\end{align}

We let $R\geq A_{\varepsilon_{1}}\sqrt{\tau_{0}}$ for $\forall\tau_{0}\in[0,\tau_{1}]$
and choose $R_{1}$ as $R_{1}:=\frac{\delta}{M_{0}^{\lambda}}R$,
then a direct computation yields 
\begin{align}
 & \frac{R_{1}^{2}}{R^{2}}\int_{0}^{\tau_{0}}(\tau_{0}-\tau)||\nabla v^{\lambda}||_{L^{2}}^{2}\rho(v^{\lambda}(\tau),R)d\tau\label{additional error}\\
 & \leq2(\frac{\delta}{M_{0}^{\lambda}})^{2}\frac{(M_{0}^{\lambda})^{2}}{\varepsilon}\int_{0}^{\tau_{0}}(\tau_{0}-\tau)||\nabla v^{\lambda}||_{L^{2}}^{2}d\tau\nonumber \\
 & \leq2\frac{\delta^{2}}{\varepsilon}\int_{0}^{\tau_{0}}(\tau_{0}-\tau)||\nabla v^{\lambda}||_{L^{2}}^{2}d\tau.\nonumber 
\end{align}

Now we can choose $\delta$ small enough such that $2\frac{\delta^{2}}{\varepsilon}\ll\frac{C_{1}'}{2}$.
For the sake of simplicity, if we define $\delta(\varepsilon)=\varepsilon$
and $\varepsilon$ small enough, substitute (\ref{additional error})
into (\ref{local virial identity-1}) and we get 
\begin{align}
 & \int\psi_{R}|v^{\lambda}(\tau_{0})|^{2}+\frac{C'_{1}}{2}\int_{0}^{\tau_{0}}(\tau_{0}-\tau)||\nabla v^{\lambda}(\tau)||_{L^{2}}^{2}d\tau\label{local virial identity-2}\\
 & \lesssim\int\psi_{R}|v^{\lambda}(0)|^{2}+\tau_{0}[Im(\int\nabla\psi_{R}(x)\nabla v^{\lambda}(0)\bar{v^{\lambda}}(0)dx)+E^{\lambda}(v_{0}^{\lambda})\tau_{0}]\nonumber \\
 & +\frac{1}{R}\int_{0}^{\tau_{0}}(\tau_{0}-\tau)[\rho(v^{\lambda}(\tau),R)+\rho(v^{\lambda}(\tau),R)^{3}+\rho(v^{\lambda}(\tau),R_{1})\rho(v^{\lambda}(\tau),R)]d\tau.\nonumber 
\end{align}

We directly estimate the right hand side of (\ref{local virial identity-2})
except the second term:
\begin{equation}
\int\psi_{R}|v^{\lambda}(\tau_{0})|^{2}\leq(\int(\psi_{R})^{3}dx)^{\frac{1}{3}}(\int|v_{0}^{\lambda}|^{3}dx)^{\frac{2}{3}}\lesssim R^{3}(M_{0}^{\lambda})^{2},\label{term 1}
\end{equation}

\begin{align}
 & \frac{1}{R}\int_{0}^{\tau_{0}}(\tau_{0}-\tau)[\rho(v^{\lambda}(\tau),R)+\rho(v^{\lambda}(\tau),R)^{3}+\rho(v^{\lambda}(\tau),R_{1})\rho(v^{\lambda}(\tau),R)]d\tau\label{term3,4,5}\\
 & \leq\frac{\tau_{0}^{2}}{R}(\frac{(M_{0}^{\lambda})^{6}}{\varepsilon^{3}}+\frac{(M_{0}^{\lambda})^{7}}{\varepsilon^{11}})\nonumber \\
 & \leq2\frac{\tau_{0}^{2}}{R}\frac{(M_{0}^{\lambda})^{7}}{\varepsilon^{11}}.\nonumber 
\end{align}

For the second term, we claim the key estimate: $\forall\tau_{0}\in[0,\tau_{1}],\;\forall A\geq A_{\varepsilon_{1}},$
let $R=A\sqrt{\tau_{0}}$ , then 
\begin{equation}
Im\int\nabla\psi_{R}\cdotp\nabla v^{\lambda}(0)\overline{v^{\lambda}(0)}+E^{\lambda}(v_{0}^{\lambda})\tau_{0}\leq C\frac{(M_{0}^{\lambda})^{2}A^{3}}{\varepsilon^{\frac{2}{3}}}\tau_{0}^{\frac{1}{2}}.\label{term 2}
\end{equation}

In step 2 and step 3, we assume (\ref{term 2}) holds and derive the
desired upper bounds (\ref{upper M-C norm-1-1}), (\ref{upper M-C norm-2-1})
and (\ref{upper H1 norm-2}). In step 4, we give the proof of (\ref{term 2})
and thus finish the whole proof. 

Since $R\geq A_{\varepsilon_{1}}\sqrt{\tau_{0}}$ and the definition,
$R_{1}=\frac{\varepsilon}{M_{0}^{\lambda}}R$, we deduce $R_{1}\geq A_{\varepsilon_{2}}\sqrt{\tau_{0}}$. 

We divide (\ref{local virial identity-2}) by $R^{3}$, and get 
\begin{equation}
\frac{1}{R}\int_{R\leq|x|\leq2R}|v^{\lambda}(\tau_{0})|^{2}dx\lesssim(M_{0}^{\lambda})^{2}+\frac{(M_{0}^{\lambda})^{2}}{\varepsilon^{\frac{2}{3}}}+\frac{(M_{0}^{\lambda})^{7}}{\varepsilon^{11}}\frac{1}{A_{\varepsilon_{1}}^{4}}\leq\frac{(M_{0}^{\lambda})^{2}}{\varepsilon},\label{M-C norm at R}
\end{equation}

by the definition of $G_{\varepsilon}$ and $A_{\varepsilon_{1}}$.

Similarly, we divide (\ref{local virial identity-2}) by $R_{1}^{3}$,
and get 
\begin{align}
\frac{1}{R_{1}}\int_{R_{1}\leq|x|\leq2R_{1}}|v^{\lambda}(\tau_{0})|^{2}dx & \leq(\frac{R}{R_{1}})^{3}(M_{0}^{\lambda})^{2}+\frac{(M_{0}^{\lambda})^{5}}{\varepsilon^{3+\frac{2}{3}}}+\frac{(M_{0}^{\lambda})^{7}}{\varepsilon^{11}}\frac{1}{A_{\varepsilon_{1}}A_{\varepsilon_{2}}^{3}}\label{M-C norm at R1}\\
 & \leq\frac{(M_{0}^{\lambda})^{5}}{\varepsilon^{3}}+\frac{(M_{0}^{\lambda})^{5}}{\varepsilon^{3+\frac{2}{3}}}+\frac{(M_{0}^{\lambda})^{7}}{\varepsilon^{11}}\frac{1}{A_{\varepsilon_{1}}A_{\varepsilon_{2}}^{3}}\nonumber \\
 & \leq\frac{(M_{0}^{\lambda})^{5}}{\varepsilon^{10}}.\nonumber 
\end{align}

By standard continuity argument, we conclude (\ref{upper M-C norm})
holds. We should note that the choice of $C_{1}$ and $\alpha_{1}$
is independent of $\lambda$ from (\ref{upper M-C norm-1}) and the
definition of $A_{\varepsilon_{1}}$. 

Step 3. Self similar decay of the gradient. 

We substitute (\ref{upper M-C norm-1-1}), (\ref{upper M-C norm-2-1})
and (\ref{term 2}) at $A=A_{\varepsilon_{1}}$ into (\ref{local virial identity-2}),
then we derive
\begin{align}
 & \int_{0}^{\tau_{0}}(\tau_{0}-\tau)||\nabla v^{\lambda}(\tau)||_{L^{2}}^{2}d\tau\\
 & \lesssim R^{3}(M_{0}^{\lambda})^{2}+C\frac{(M_{0}^{\lambda})^{2}A_{\varepsilon_{1}}^{3}}{\varepsilon^{\frac{2}{3}}}\tau_{0}^{\frac{3}{2}}+\frac{\tau_{0}^{2}}{R}\frac{(M_{0}^{\lambda})^{7}}{\varepsilon^{11}}\nonumber \\
 & \leq\frac{\tau_{0}^{\frac{3}{2}}(M_{0}^{\lambda})^{2}A_{\varepsilon_{1}}^{3}}{\varepsilon^{\frac{2}{3}}}\leq\frac{G_{\varepsilon}}{2}\tau_{0}^{\frac{3}{2}},\nonumber 
\end{align}

which finishes the proof of (\ref{upper H1 norm}). From (\ref{chosen parameter}),
we know the choice of $\alpha_{2}$ in (\ref{upper H1 norm}) is also
independent of $\lambda$. 

Step 4. Proof of the momentum estimate (\ref{term 2}).

We can not directly use the interpolation estimate,
\begin{equation}
|Im(\int\nabla\psi_{R}(x)\nabla v^{\lambda}(0)\bar{v^{\lambda}}(0)dx|\lesssim||v_{0}^{\lambda}||_{\dot{H}^{\frac{1}{2}}}^{2}R,
\end{equation}

since we only know the information about $||v_{0}^{\lambda}||_{L^{3}}$
rather than $||v_{0}^{\lambda}||_{\dot{H}^{\frac{1}{2}}}$. On the
other hand, if we use H\"older inequality to estimate 
\begin{align*}
|Im(\int\nabla\psi_{R}(x)\nabla v^{\lambda}(0)\bar{v^{\lambda}}(0)dx)| & \leq R^{\frac{3}{2}}||\nabla v_{0}^{\lambda}||_{L^{2}}(\frac{1}{R^{3}}\int\psi_{R}|v_{0}^{\lambda}|^{2})^{\frac{1}{2}}\\
 & \leq R^{\frac{3}{2}}||\nabla v_{0}^{\lambda}||_{L^{2}}M_{0}^{\lambda}.
\end{align*}

Although we have used the information of $||v_{0}^{\lambda}||_{L^{3}}$,
we can not take full advantage of the information from (\ref{upper H1 norm})
at the initial time $\tau=0$. It is because that we hope $||\nabla v^{\lambda}(\tilde{\tau})||_{L^{2}}$
has the asymptotic behaviour with 
\begin{equation}
||\nabla v^{\lambda}(\tilde{\tau})||_{L^{2}}\leq\frac{C}{\tilde{\tau}^{\frac{1}{4}}},\label{H1 decay}
\end{equation}

for some special large enough time $\tilde{\tau}$ from (\ref{upper H1 norm}).
With the above analysis, a natural idea is to choose a time $\tilde{\tau}$
carefully to obtain the decay estimate (\ref{H1 decay}) at $\tilde{\tau}$.
Then we use (\ref{virial identity-1}) to conclude the difference
between $Im(\int\nabla\psi_{R}(x)\nabla v^{\lambda}(0)\bar{v^{\lambda}}(0)dx$
and $Im(\int\nabla\psi_{R}(x)\nabla v^{\lambda}(\tilde{\tau})\bar{v^{\lambda}}(\tilde{\tau})dx$
is suitably small, which helps us to finish the proof (\ref{term 2}).
In order to explain the above content more clearly, we have divided
the proof into three parts.

Part I: Choose a suitable time to obtain the decay estimate (\ref{H1 decay}).

Let $\tau_{0}\in[0,\tau_{1}]\;A\geq A_{\varepsilon_{1}}$ and $R=A\sqrt{\tau_{0}}$.
First, we should choose the proper $\tilde{\tau}$ as above. To be
more specific, we claim the following fact: there exists a universal
constant $K>0$ which is independent of $\lambda$ and $\tilde{\tau}_{0}$
such that 
\begin{equation}
\tilde{\tau}_{0}\in[\frac{\varepsilon^{\frac{2}{3}}}{4}\tau_{0},\frac{\varepsilon^{\frac{2}{3}}}{2}\tau_{0}]\;with\;||\nabla v^{\lambda}(\tilde{\tau}_{0})||_{L^{2}}^{2}\leq\frac{KG_{\varepsilon}}{\tilde{\tau}_{0}^{\frac{1}{2}}}.\label{special time H1 decay}
\end{equation}

Proof of (\ref{special time H1 decay}). By contradiction, let $\tilde{\tau}=\varepsilon^{\frac{2}{3}}\tau_{0}$,
then
\begin{equation}
\int_{\frac{\tilde{\tau}}{4}}^{\frac{\tilde{\tau}}{2}}||\nabla v^{\lambda}(\sigma)||_{L^{2}}^{2}d\sigma\geq KG_{\varepsilon}\int_{\frac{\tilde{\tau}}{4}}^{\frac{\tilde{\tau}}{2}}\frac{d\sigma}{\sigma^{\frac{1}{2}}}\geq CKG_{\varepsilon}\tilde{\tau}^{\frac{1}{2}}.
\end{equation}

Moreover, $\tilde{\tau}=\varepsilon^{\frac{2}{3}}\tau_{0}\leq\tau_{0}\leq\tau_{1}$
and (\ref{upper H1 norm-1}) implies:
\begin{equation}
G_{\varepsilon}\tilde{\tau}^{\frac{3}{2}}\geq\int_{0}^{\tilde{\tau}}(\tilde{\tau}-\sigma)||\nabla v^{\lambda}(\sigma)||_{L^{2}}^{2}d\sigma\geq\frac{\tilde{\tau}}{2}\int_{\frac{\tilde{\tau}}{4}}^{\frac{\tilde{\tau}}{2}}||\nabla v^{\lambda}(\sigma)||_{L^{2}}^{2}d\sigma\geq CKG_{\varepsilon}\tilde{\tau}^{\frac{3}{2}}.\label{contradiction}
\end{equation}

From (\ref{contradiction}), we conclude a contradiction for $K>0$
large enough.

Part II: Derive the desired upper bound for the quantity $Im(\int\nabla\psi_{R}(x)\nabla v^{\lambda}(\tilde{\tau}_{0})\bar{v^{\lambda}}(\tilde{\tau}_{0})dx$.

Define $R=A\sqrt{\tau_{0}}=A_{1}\sqrt{\tilde{\tau}_{0}}$ and thus
\begin{equation}
\frac{\varepsilon^{\frac{1}{3}}}{16}\leq\frac{A}{A_{1}}\leq\varepsilon^{\frac{1}{3}}.\label{choose A1}
\end{equation}

We claim:
\begin{equation}
|Im\int\nabla\psi_{R}\cdotp\nabla v^{\lambda}(\tilde{\tau}_{0})\overline{v^{\lambda}(\tilde{\tau}_{0})}|\leq C\frac{(M_{0}^{\lambda})^{2}A^{3}}{\varepsilon^{\frac{2}{3}}}\tau_{0}^{\frac{1}{2}}.\label{special time momentum}
\end{equation}

Proof of (\ref{special time momentum}). Before we derive (\ref{special time momentum}),
we show the following inequality
\begin{equation}
\frac{1}{R^{3}}\int\psi_{R}|v^{\lambda}(\tilde{\tau}_{0})|^{2}\leq C[(M_{0}^{\lambda})^{2}+\frac{G_{\varepsilon}}{A_{1}^{3}}].\label{weighted L2-2}
\end{equation}

In fact, from (\ref{virial identity-1}) and Cauchy-Schwarz inequality,
we obtain
\[
|\frac{d}{d\tau}\int\psi_{R}|v^{\lambda}|^{2}|=2|Im(\int\nabla\psi_{R}\cdotp\nabla v^{\lambda}\bar{v^{\lambda}})|\leq C||\nabla v^{\lambda}||_{L^{2}}(\int\psi_{R}|v^{\lambda}|^{2})^{\frac{1}{2}}.
\]

Therefore, we integrate this differential inequality from $0$ to
$\tilde{\tau}_{0}$ and get:
\begin{align}
\int\psi_{R}|v^{\lambda}(\tilde{\tau}_{0})|^{2} & \leq C[\int\psi_{R}|v^{\lambda}(0)|^{2}+(\int_{0}^{\tilde{\tau}_{0}}||\nabla v^{\lambda}(\sigma)||_{L^{2}}d\sigma)^{2}]\label{weighted L2-3}\\
 & \leq C[R^{3}(M_{0}^{\lambda})^{2}+\tilde{\tau}_{0}\int_{0}^{\tilde{\tau}_{0}}||\nabla v^{\lambda}(\sigma)||_{L^{2}}^{2}d\sigma].\nonumber 
\end{align}

Since $2\tilde{\tau}_{0}\leq\varepsilon^{\frac{2}{3}}\tau_{0}\leq\tau_{0}\leq\tau_{1}$
and using (\ref{upper H1 norm-1}), we obtain
\begin{equation}
\int_{0}^{\tilde{\tau}_{0}}||\nabla v^{\lambda}(\sigma)||_{L^{2}}^{2}d\sigma\leq\frac{1}{\tilde{\tau}_{0}}\int_{0}^{2\tilde{\tau}_{0}}(2\tilde{\tau}_{0}-\sigma)||\nabla v^{\lambda}(\sigma)||_{L^{2}}^{2}d\sigma\leq CG_{\varepsilon}\tilde{\tau}_{0}^{\frac{1}{2}}.\label{H1 decay-2}
\end{equation}

Thus, recalling that $R=A_{1}\sqrt{\tilde{\tau}_{0}}$, we combine
(\ref{weighted L2-3}) and (\ref{H1 decay-2}) to get
\[
\int\psi_{R}|v^{\lambda}(\tilde{\tau}_{0})|^{2}\leq C[R^{3}(M_{0}^{\lambda})^{2}+G_{\varepsilon}\tilde{\tau}_{0}^{\frac{3}{2}}]=CR^{3}[(M_{0}^{\lambda})^{2}+\frac{G_{\varepsilon}}{A_{1}^{3}}]
\]

and concludes the proof of (\ref{weighted L2-2}).

Now we control the virial quantity (\ref{special time momentum})
at time $\tilde{\tau}_{0}$:
\begin{align*}
|Im\int\nabla\psi_{R}\cdotp\nabla v^{\lambda}(\tilde{\tau}_{0})\overline{v^{\lambda}(\tilde{\tau}_{0})}| & \leq R^{\frac{3}{2}}||\nabla v^{\lambda}(\tilde{\tau}_{0})||_{L^{2}}(\frac{1}{R^{3}}\int|v^{\lambda}(\tilde{\tau}_{0})|^{2})^{\frac{1}{2}}\\
 & \leq CR^{\frac{3}{2}}\frac{G_{\varepsilon}^{\frac{1}{2}}}{\tilde{\tau}_{0}^{\frac{1}{4}}}[(M_{0}^{\lambda})^{2}+\frac{G_{\varepsilon}}{A_{1}^{3}}]^{\frac{1}{2}}\\
 & \leq CAA_{1}^{\frac{1}{2}}\tau_{0}^{\frac{1}{2}}G_{\varepsilon}^{\frac{1}{2}}[M_{0}^{\lambda}+\frac{G_{\varepsilon}^{\frac{1}{2}}}{A_{1}^{\frac{3}{2}}}]
\end{align*}

from $A\geq A_{\varepsilon_{1}}$ and the choice of $\tilde{\tau}_{0}$.
From the definition of $G_{\varepsilon},A_{\varepsilon_{1}}$ and
$A_{1}$ in (\ref{chosen parameter}) and (\ref{choose A1}), we derive
\begin{align*}
 & |Im\int\nabla\psi_{R}\cdotp\nabla v^{\lambda}(\tilde{\tau}_{0})\overline{v^{\lambda}(\tilde{\tau}_{0})}|\\
 & \leq C(M_{0}^{\lambda})^{2}A^{3}\tau_{0}^{\frac{1}{2}}\frac{1}{\varepsilon^{\frac{1}{2}}\cdotp\varepsilon^{\frac{1}{6}}}=\frac{C(M_{0}^{\lambda})^{2}A^{3}}{\varepsilon^{\frac{2}{3}}}\tau_{0}^{\frac{1}{2}},
\end{align*}

which finishes the proof of (\ref{special time momentum}).

Part III: Prove the initial control on the virial quantity (\ref{term 2}). 

Using (\ref{local virial identity}), we derive the crude estimate
which connects $\tau=0$ with $\tau=\tilde{\tau}_{0}$:
\begin{align}
|\frac{d}{d\tau}Im(\int\nabla\psi_{R}\cdotp\nabla v^{\lambda}\overline{v^{\lambda}})| & \leq C\{|E^{\lambda}(v_{0}^{\lambda})|+\int|\nabla v^{\lambda}|^{2}dx\label{term2-1}\\
 & +R^{-1}[\rho(v^{\lambda}(\tau),R)+\rho(v^{\lambda}(\tau),R)^{3}+\rho(v^{\lambda}(\tau),R_{1})\rho(v^{\lambda}(\tau),R)]\}.\nonumber 
\end{align}

Since 
\begin{align*}
 & R^{-1}[\rho(v^{\lambda}(\tau),R)+\rho(v^{\lambda}(\tau),R)^{3}+\rho(v^{\lambda}(\tau),R_{1})\rho(v^{\lambda}(\tau),R)]\\
 & \leq\frac{C}{R}\frac{(M_{0}^{\lambda})^{7}}{\varepsilon^{11}}\leq\frac{1}{\tilde{\tau}_{0}^{\frac{1}{2}}}
\end{align*}

for $\varepsilon$ small enough, we conclude from (\ref{term2-1})
\begin{equation}
|\frac{d}{d\tau}Im(\int\nabla\psi_{R}\cdotp\nabla v^{\lambda}\overline{v^{\lambda}})|\leq C\{|E^{\lambda}(v_{0}^{\lambda})|+\int|\nabla v^{\lambda}|^{2}dx+\frac{1}{\tilde{\tau}_{0}^{\frac{1}{2}}}\}.\label{term2-2}
\end{equation}

We integrate (\ref{term2-2}) in time from 0 to $\tilde{\tau}_{0}$
and get 
\begin{align}
|Im\int\nabla\psi_{R}\cdotp\nabla v^{\lambda}(0)\overline{v^{\lambda}(0)}| & \leq|Im\int\nabla\psi_{R}\cdotp\nabla v^{\lambda}(\tilde{\tau}_{0})\overline{v^{\lambda}(\tilde{\tau}_{0})}|\\
 & +C[\int_{0}^{\tilde{\tau}_{0}}||\nabla v^{\lambda}(\sigma)||_{L^{2}}^{2}d\sigma+|E^{\lambda}(v_{0}^{\lambda})|\tilde{\tau}_{0}+\tilde{\tau}_{0}^{\frac{1}{2}}].\nonumber 
\end{align}

Thanks to (\ref{special time momentum}) and (\ref{H1 decay-2}),
we conclude
\begin{align}
Im\int\nabla\psi_{R}\cdotp\nabla v^{\lambda}(0)\overline{v^{\lambda}(0)}+E^{\lambda}(v_{0}^{\lambda})\tau_{0} & \leq\frac{C(M_{0}^{\lambda})^{2}A^{3}}{\varepsilon^{\frac{2}{3}}}\tau_{0}^{\frac{1}{2}}+CG_{\varepsilon}\tilde{\tau}_{0}^{\frac{1}{2}}\label{term2-3}\\
 & +E^{\lambda}(v_{0})\tau_{0}+C|E^{\lambda}(v_{0})|\tilde{\tau}_{0}.\nonumber 
\end{align}

Our main task is to estimate the right-hand side in (\ref{term2-3})
term by term. Since the first term has been controlled, we estimate
the second term
\begin{equation}
G_{\varepsilon}\tilde{\tau}_{0}^{\frac{1}{2}}\leq CA_{\varepsilon_{1}}^{3}(M_{0}^{\lambda})^{2}\frac{\varepsilon^{\frac{1}{3}}}{\varepsilon}\tau_{0}^{\frac{1}{2}}\leq\frac{C(M_{0}^{\lambda})^{2}A^{3}}{\varepsilon^{\frac{2}{3}}}\tau_{0}^{\frac{1}{2}},
\end{equation}

where we use the definition of $G_{\varepsilon},A_{\varepsilon_{1}}$
in (\ref{chosen parameter}) and the choice of $\tilde{\tau}_{0}$
in (\ref{special time H1 decay}). Lastly, to control the remaining
two terms in (\ref{term2-3}), we observe that
\begin{align*}
E^{\lambda}(v_{0}^{\lambda})\tau_{0}+C|E^{\lambda}(v_{0}^{\lambda})|\tilde{\tau}_{0} & \leq[E^{\lambda}(v_{0}^{\lambda})+C\varepsilon^{\frac{2}{3}}|E^{\lambda}(v_{0}^{\lambda})|]\tau_{0}\\
 & \leq C\max[E^{\lambda}(v_{0}^{\lambda}),0]\tau_{0}\leq C\tau_{0}^{\frac{1}{2}}
\end{align*}

for $\varepsilon>0$ small enough and we have used the condition (\ref{scaling energy-3}).
Summing up the above estimates, we finish the desired proof (\ref{term 2}).
\end{proof}
The following proposition implies a nontrivial repartition of the $L^{2}$
mass of the initial data. This repartition helps us to get uniform
lower bound on sufficient disjoint annuli.
\begin{prop}
\label{prop 2 of ours}(Lower bound on a weighted local $L^{2}$ norm
of $v^{\lambda}(0))$

Let $\tau_{*}>0$ and $v^{\lambda}(\tau)\in C([0,\tau_{*}],\dot{H}^{\frac{1}{2}}\cap\dot{H}^{1})$
be a radially symmetric solution to (\ref{eq:scaling-3-1}) and assume
(\ref{scaling energy-3}) and (\ref{scaling initial data-3}) in Prop
\ref{prop 1 of ours} hold. Then there exist universal constants
$\alpha_{3},c_{3}>0$ which are independent of $\lambda$ such that
the following holds true. Let 

\begin{equation}
\tilde{\lambda}_{v}(\tau)=(\frac{1}{||\nabla v^{\lambda}(\tau)||_{L^{2}}})^{2},\label{scaling H1}
\end{equation}

and let 
\begin{equation}
\tau_{0}\in[0,\frac{\tau_{*}}{2}],
\end{equation}

\begin{equation}
\tilde{\lambda}_{v}(\tau_{0})E^{\lambda}(v_{0}^{\lambda})\leq\frac{1}{4}.\label{scaling energy-4}
\end{equation}

Let 
\begin{equation}
F_{*}=\frac{\sqrt{\tau_{0}}}{\tilde{\lambda}_{v}(\tau_{0})},\label{scaling time}
\end{equation}

and 
\begin{equation}
D_{*}=(M_{0}^{\lambda})^{\alpha_{3}}max[1,F_{*}^{3}],\label{large constant}
\end{equation}

then 
\begin{equation}
\frac{1}{\tilde{\lambda}_{v}(\tau_{0})}\int_{|x|\leq D_{*}\tilde{\lambda}_{v}(\tau_{0})}|v^{\lambda}(0)|^{2}\geq c_{3}.\label{weighted L2 lower bound}
\end{equation}
\end{prop}
\begin{proof}
Step 1. Energy constraint and lower bound on $v^{\lambda}(\tau_{0}).$

We claim that there exist universal constants $C_{3},c_{3}>0$ which
are independent of $\lambda$ such that:

\begin{equation}
\frac{1}{\tilde{\lambda}_{v}(\tau_{0})}\int_{|x|\leq A_{*}\tilde{\lambda}_{v}(\tau_{0})}|v^{\lambda}(\tau_{0})|^{2}\geq c_{3}\label{weighted L2 at large time}
\end{equation}

with
\begin{equation}
A_{*}=C_{3}\max[(M_{0}^{\lambda})^{\alpha_{1}}F_{*},(M_{0}^{\lambda})^{6}].\label{large constant-1}
\end{equation}

Now we prove the claim (\ref{weighted L2 at large time}). Consider
a renormalization of $v^{\lambda}(\tau_{0})$ :

\begin{equation}
w(x)=\tilde{\lambda}_{v}(\tau_{0})v^{\lambda}(\tau_{0},\tilde{\lambda}_{v}(\tau_{0})x),\label{renorm of v}
\end{equation}

then from (\ref{scaling H1}), (\ref{scaling energy-4}) and the conservation
of the energy:
\begin{equation}
||\nabla w||_{L^{2}}=1
\end{equation}

and
\begin{align}
E^{\lambda\cdotp\tilde{\lambda}_{v}(\tau_{0})}(w) & =\frac{1}{2}\int|\nabla w|^{2}dx-\frac{1}{4}\int(V_{\lambda\cdotp\tilde{\lambda}_{v}(\tau_{0})}*|w|^{2})(x)|w(x)|^{2}dx\\
 & =\tilde{\lambda}_{v}(\tau_{0})[\frac{1}{2}\int|\nabla v^{\lambda}|^{2}dx-\frac{1}{4}\int(V_{\lambda}*|v^{\lambda}|^{2})(x)|v^{\lambda}(x)|^{2}dx]\nonumber \\
 & =\tilde{\lambda}_{v}(\tau_{0})E^{\lambda}(v^{\lambda}(\tau_{0}))=\tilde{\lambda}_{v}(\tau_{0})E^{\lambda}(v{}_{0}^{\lambda})\leq\frac{1}{4},\nonumber 
\end{align}

and thus 
\[
\iint_{R^{3}\times R^{3}}V_{\lambda\cdotp\tilde{\lambda}_{v}(\tau_{0})}(x-y)|w(x)|^{2}|w(y)|^{2}dxdy=4(\frac{1}{2}||\nabla w||_{L_{x}^{2}}^{2}-E^{\lambda\cdotp\tilde{\lambda}_{v}(\tau_{0})}(w))\geq1.
\]

Pick now $\varepsilon>0$ small enough and let 
\begin{equation}
A_{\varepsilon}=C_{\varepsilon}\max[(M_{0}^{\lambda})^{\alpha_{1}}F_{*},(M_{0}^{\lambda})^{6}]
\end{equation}

for $C_{\varepsilon}$ large enough to be chosen. First we observe
that 
\begin{align*}
\rho(w,A_{\varepsilon}) & \leq\rho(w,(M_{0}^{\lambda})^{\alpha_{1}}F_{*})=\rho(v^{\lambda}(\tau_{0}),(M_{0}^{\lambda})^{\alpha_{1}}\tilde{\lambda}_{v}(\tau_{0})F_{*})\\
 & =\rho(v^{\lambda}(\tau_{0}),(M_{0}^{\lambda})^{\alpha_{1}}\sqrt{\tau_{0}})\leq C_{1}(M_{0}^{\lambda})^{2}.
\end{align*}

Thus, from (\ref{the integrability condition}), (\ref{eq:sobolev-2})
\begin{align}
 & \iint_{|x|\geq A_{\varepsilon},|y|\geq A_{\varepsilon}}V_{\lambda\cdotp\tilde{\lambda}_{v}(\tau_{0})}(x-y)|w(x)|^{2}|w(y)|^{2}dxdy\leq||w(x)\cdotp\chi_{\{|x|\geq A_{\varepsilon}\}}||_{L^{4}}^{2}||w(y)\cdotp\chi_{\{|y|\geq A_{\varepsilon}\}}||_{L^{4}}^{2}\\
 & \leq||w(x)\cdotp\chi_{\{|x|\geq A_{\varepsilon}\}}||_{L^{4}}^{4}\nonumber \\
 & \leq\varepsilon||\nabla w||_{L_{x}^{2}}^{2}+\frac{C(\varepsilon)}{A_{\varepsilon}}[\rho(w,A_{\varepsilon})+\rho(w,A_{\varepsilon})^{3}]\nonumber \\
 & \leq2\varepsilon\nonumber 
\end{align}

for $C_{\varepsilon}$ large enough and 
\begin{equation}
\iint_{|x|\leq A_{\varepsilon},|y|\leq A_{\varepsilon}}+\iint_{|x|\geq A_{\varepsilon},|y|\leq A_{\varepsilon}}+\iint_{|x|\leq A_{\varepsilon},|y|\geq A_{\varepsilon}}V_{\lambda\cdotp\tilde{\lambda}_{v}(\tau_{0})}(x-y)|w(x)|^{2}|w(y)|^{2}dxdy\geq\frac{1}{2}.
\end{equation}

By the Pigeon house principle, there must exist at least one term that
is bigger than $\frac{1}{6}$. Without loss of generality, we assume
\[
\iint_{|x|\leq A_{\varepsilon},|y|\geq A_{\varepsilon}}V_{\lambda\cdotp\tilde{\lambda}_{v}(\tau_{0})}(x-y)|w(x)|^{2}|w(y)|^{2}dxdy\geq\frac{1}{6}.
\]

Then by the H\"older inequality, Young inequality and Sobolev embedding
theory, 
\[
\frac{1}{6}\leq C||\chi_{\{|x|\leq A_{\varepsilon}\}}w||_{L_{x}^{2}}||\nabla w||_{L_{y}^{2}}^{3},
\]

which implies that 
\[
\int_{|x|\leq A_{\varepsilon}}|w(x)|^{2}dx\geq c_{3}>0
\]

for some constant $c_{3}>0$ which is independent of $\lambda$. By
(\ref{renorm of v}), this is 
\[
\frac{1}{\tilde{\lambda}_{v}(\tau_{0})}\int_{|x|\leq A_{*}\tilde{\lambda}_{v}(\tau_{0})}|v^{\lambda}(\tau_{0})|^{2}\geq c_{3}
\]

which finishes the claim (\ref{weighted L2 at large time}).

Step 2. Backwards integration of the $L^{2}$ fluxes.

We claim: $\forall\varepsilon>0$, there exists $\tilde{C}_{\varepsilon}>0$
such that $\forall D\geq D_{\varepsilon}$ with 
\begin{equation}
D_{\varepsilon}=\tilde{C}_{\varepsilon}\max[F_{*},F_{*}^{3}]\cdotp\max[(M_{0}^{\lambda})^{\alpha_{1}},(M_{0}^{\lambda})^{2+\alpha_{2}}],\label{large constant-2}
\end{equation}

let
\begin{equation}
\tilde{R}=\tilde{R}(D,\tau_{0})=D\tilde{\lambda}_{v}(\tau_{0}),\label{renorm of time-1}
\end{equation}

and $\chi_{\tilde{R}}(r)=\chi(\frac{r}{R})$ for some smooth radially
symmetric cut-off function $\chi(r)=1$ for $r\leq1,\chi(r)=0$ for
$r\geq2$, then :
\begin{equation}
|\frac{1}{\tilde{\lambda}_{v}(\tau_{0})}\int\chi_{\tilde{R}}|v^{\lambda}(\tau_{0})|^{2}-\frac{1}{\tilde{\lambda}_{v}(\tau_{0})}\int\chi_{\tilde{R}}|v^{\lambda}(0)|^{2}|<\varepsilon.\label{the difference}
\end{equation}

(\ref{weighted L2 at large time}) and (\ref{the difference}) now
imply (\ref{weighted L2 lower bound}).

Proof of (\ref{the difference}). Pick $\varepsilon>0$. We compute
the $L^{2}$ fluxes from (\ref{virial identity-1}) with $\chi_{\tilde{R}}$
:
\begin{align*}
|\frac{d}{d\tau}\int\chi_{\tilde{R}}|v^{\lambda}|^{2}| & =2|Im(\int\nabla\chi_{\tilde{R}}\cdotp\nabla v^{\lambda}\bar{v^{\lambda}})|\\
 & \leq\frac{C}{\tilde{R}}||\nabla v^{\lambda}(\tau)||_{L^{2}}(\int_{\tilde{R}\leq|x|\leq2\tilde{R}}|v^{\lambda}(\tau)|^{2})^{\frac{1}{2}}\\
 & \leq\frac{C}{\tilde{R}^{\frac{1}{2}}}||\nabla v^{\lambda}(\tau)||_{L^{2}}(\rho(v^{\lambda}(\tau),\tilde{R}))^{\frac{1}{2}}.
\end{align*}

Now observe from (\ref{scaling time}), (\ref{large constant-2})
and (\ref{renorm of time-1}) that :
\[
\forall\tau\in[0,\tau_{0}],\;\tilde{R}=D\tilde{\lambda}_{v}(\tau_{0})\geq D_{\varepsilon}\frac{\tilde{\lambda}_{v}(\tau_{0})}{\sqrt{\tau_{0}}}\sqrt{\tau_{0}}\geq\frac{D_{\varepsilon}}{F_{*}}\sqrt{\tau}\geq(M_{0}^{\lambda})^{\alpha_{1}}\sqrt{\tau},
\]

and thus (\ref{upper M-C norm}) and the monotonicity of $\rho$ ensure
:
\begin{equation}
\forall\tau\in[0,\tau_{0}],\;\rho(v^{\lambda}(\tau),\tilde{R})\leq\rho(v^{\lambda}(\tau),(M_{0}^{\lambda})^{\alpha_{1}}\sqrt{\tau})<C_{1}(M_{0}^{\lambda})^{2}.
\end{equation}

Now we derive that 
\[
\forall\tau\in[0,\tau_{0}],\;|\frac{d}{d\tau}\int\chi_{\tilde{R}}|v^{\lambda}(\tau)|^{2}|\leq\frac{CM_{0}^{\lambda}}{\tilde{R}^{\frac{1}{2}}}||\nabla v^{\lambda}(\tau)||_{L^{2}}.
\]

We integrate this between 0 and $\tau_{0}$, divide by $\tilde{R}$
and use (\ref{upper H1 norm}) to get :
\begin{align*}
 & |\frac{1}{\tilde{\lambda}_{v}(\tau_{0})}\int\chi_{\tilde{R}}|v^{\lambda}(\tau_{0})|^{2}-\frac{1}{\tilde{\lambda}_{v}(\tau_{0})}\int\chi_{\tilde{R}}|v^{\lambda}(0)|^{2}|\\
 & \leq\frac{CM_{0}^{\lambda}\cdotp D}{\tilde{R}^{\frac{3}{2}}}\int_{0}^{\tau_{0}}||\nabla v^{\lambda}(\tau)||_{L^{2}}d\tau\leq\frac{CM_{0}^{\lambda}\cdotp D}{D^{\frac{3}{2}}\cdotp\tilde{\lambda}_{v}(\tau_{0})^{\frac{3}{2}}}(\tau_{0}\int_{0}^{\tau_{0}}||\nabla v^{\lambda}(\tau)||_{L^{2}}^{2}d\tau)^{\frac{1}{2}}\\
 & \leq\frac{CM_{0}^{\lambda}}{D^{\frac{1}{2}}\cdotp\tilde{\lambda}_{v}(\tau_{0})^{\frac{3}{2}}}(\int_{0}^{2\tau_{0}}(2\tau_{0}-\tau)||\nabla v(\tau)||_{L^{2}}^{2}d\tau)^{\frac{1}{2}}\leq\frac{CM_{0}^{\lambda}}{D^{\frac{1}{2}}\cdotp\tilde{\lambda}_{v}(\tau_{0})^{\frac{3}{2}}}(M_{0}^{\lambda})^{\frac{\alpha_{2}}{2}}\cdotp\tau_{0}{}^{\frac{3}{4}}\\
 & \leq\frac{C(M_{0}^{\lambda})^{1+\frac{\alpha_{2}}{2}}F_{*}^{\frac{3}{2}}}{D_{\varepsilon}^{\frac{1}{2}}}\leq\varepsilon
\end{align*}

for $\tilde{C}_{\varepsilon}$ large enough which ends the proof. 
\end{proof}
In the end of this section, we state a proposition as a blackbox to
finish the proof of Theorem \ref{thm 2 of ours},
\begin{prop}
\label{blackbox in Hartree}Let $v_{\varepsilon}(\tau,x)\in C([0,e^{N}],\dot{H}^{\frac{1}{2}}\cap\dot{H}^{1})$\footnote{$N$ is a large enough fixed number.}
be a radial solution to (\ref{eq:Hartree-4}) with the potential function
\[
V_{\varepsilon}(x):=\frac{1}{\varepsilon^{3}}V(\frac{x}{\varepsilon}),\;\forall\varepsilon>0,
\]

where $V(x)$ is a fixed potential function satisfying the conditions
in Theorem \ref{thm 2 of ours} and initial data satisfies 
\begin{equation}
||v_{\varepsilon}(0,x)||_{\dot{H}^{1}}=1,\label{renorm condition}
\end{equation}
Besides, we assume the following conditions\footnote{$E^{V_{\varepsilon}}(v_{0}):=\frac{1}{2}\int|\nabla v_{\varepsilon}(0)|^{2}-\frac{1}{4}\int(V_{\varepsilon}*|v_{\varepsilon}(0)|^{2})|v_{\varepsilon}(0)|^{2}.$}
\begin{equation}
e^{\frac{N}{2}}\cdotp\max(E^{V}(v_{\varepsilon}(0)),0)<1,\label{energy condition-hartree-1}
\end{equation}

\begin{equation}
M_{0}:=\frac{4||v_{\varepsilon}(0)||_{L^{3}}}{c_{V}}\geq2,\label{initial condition-hartree-1}
\end{equation}

and 
\begin{equation}
\forall\tau_{0}\in[0,e^{N}],\;\frac{E^{V_{\varepsilon}}(v_{\varepsilon}(0))}{||v_{\varepsilon}(\tau_{0})||_{\dot{H}^{1}}^{2}}\leq\frac{1}{4},\label{energy condition-hartree-2}
\end{equation}

then there exists a universal constant $\gamma_{1}>0$ which is independent
of $\varepsilon$ such that 
\[
||v_{\varepsilon}(0)||_{L^{3}}\geq N^{\gamma_{1}}.
\]
\end{prop}
Using this proposition to equation (\ref{eq:scaling-3-1}) with $\tau=0$,
it is direct to see 
\[
||u(t,x)||_{L^{3}}=||v^{\lambda}(0,x)||_{L^{3}}\geq|\log(T-t)|^{\gamma},
\]

and we finish the proof of Theorem \ref{thm 2 of ours}.

\section{The connection between Schr0dinger equation and Hartree equation}

In section 2, we have proved Theorem \ref{thm 2 of ours}. In this
section, we aim to show the our result also implies Theorem \ref{thm PR and FM}
as said in section 1.2. Before our analysis, we state a stronger version
proved by Merle and Rapha\"el in \cite{Main reference of FM and PR},
which is also claimed in Remark \ref{a weak version},
\begin{prop}
\label{blackbox of FM PR}Let $v(\tau,x)\in C([0,e^{N}],\dot{H}^{\frac{1}{2}}\cap\dot{H}^{1})$
be a radial solution to equation (\ref{eq:Schrodinger}) with the
initial data satisfying 
\begin{equation}
||v(0,x)||_{\dot{H}^{1}}=1,\label{renorm condition-1}
\end{equation}
and satisfies the following conditions\footnote{$E(v_{0}):=\frac{1}{2}\int|\nabla v_{0}|^{2}-\frac{1}{4}\int|v_{0}|^{4}.$}
\begin{equation}
e^{\frac{N}{2}}\cdotp\max(E(v_{0}),0)<1,\label{energy condition-schrodinger-1}
\end{equation}
\begin{equation}
M_{0}:=\frac{4||v_{0}||_{L^{3}}}{C_{GN}}\geq2,\label{initial condition-schrodinger-1}
\end{equation}

and 
\begin{equation}
\forall\tau_{0}\in[0,e^{N}],\;\frac{E(v_{0})}{||v(\tau_{0})||_{\dot{H}^{1}}^{2}}\leq\frac{1}{4},\label{energy condition-schrodinger-2}
\end{equation}

then there exists a universal constant $\gamma>0$ such that 
\begin{equation}
||v_{0}||_{L^{3}}\geq N^{\gamma}.\label{conclusion of FM PR in blackbox}
\end{equation}
\end{prop}
Now we want to derive proposition \ref{blackbox of FM PR} by applying
our conclusion, proposition \ref{blackbox in Hartree}. First we state
a result of standard stability theory, and supply the proof in Appendix
B,
\begin{lem}
\label{local theory lemma C}We consider the nonlinear Schr\"{o}dinger
equation (\ref{eq:Schrodinger}) and Hartree equation (\ref{eq:Hartree-4})
with the same radial initial data $v_{0}\in\dot{H}^{\frac{1}{2}}\cap\dot{H}^{1}$,
and assume the normalization condition
\begin{equation}
||V(x)||_{L^{1}}=||V_{\varepsilon}(x)||_{L^{1}}=1.
\end{equation}

We define $T_{max}$ as the lifetime of $v(\tau)$ and claim the following
holds true, $\forall\delta>0,\;\forall T\in[0,T_{max}),\;\exists\varepsilon^{*}=\varepsilon^{*}(\delta,T)>0,$
such that $\forall0<\varepsilon<\varepsilon^{*}$ , 
\begin{equation}
||v_{\varepsilon}(\tau)-v(\tau)||_{L_{\tau}^{\infty}([0,T],\dot{H}^{1})}\leq\delta.
\end{equation}
\end{lem}
Combining Lemma \ref{local theory lemma C} and Proposition \ref{blackbox in Hartree},
we will prove Proposition \ref{blackbox of FM PR}. The main task
is to verify the above conditions in Proposition \ref{blackbox of FM PR}
hold true item by item. 
\begin{proof}
From the standard analysis knowledge and the definition of $E(v_{0}),E^{V_{\varepsilon}}(v_{0})$,
$\forall m_{1}>0,$ $\exists m_{2}>0$, such that $\forall0<\varepsilon<m_{2}$,
\begin{equation}
|E(v_{0})-E^{V_{\varepsilon}}(v_{0})|<m_{1}.\label{the difference of initial energy}
\end{equation}

Since $\tau_{*}$ is a fixed number, condition (\ref{energy condition-hartree-1})
is naturally established. The assumption (\ref{energy condition-hartree-2})
is verified by condition (\ref{energy condition-schrodinger-2}).
Lastly, combining Lemma \ref{local theory lemma C} and (\ref{the difference of initial energy}),
we conclude (\ref{energy condition-hartree-2}) holds. The conclusion
of the Proposition \ref{blackbox in Hartree} then leads directly
to the conclusion (\ref{conclusion of FM PR in blackbox}) as we want.
\end{proof}
\begin{description}
\item [{Acknowledgement}] The author is very grateful to his advisers,
Ping Zhang and Chenjie Fan for their constant encouragements, guidance
and very helpful discussions during preparing this work. The author
thanks Ning Liu and Hequn Zhang, for some helpful comments and discussions. The author is supported by National Key R\&D
Program of China under grant 2021YFA1000800.
\end{description}

\appendix

\section{A virial argument and The applications of local theory}

In this section, we give complete analysis claimed in Remark \ref{rmk 2 of thm 1}
and some applications of local theory. First, through a classical
virial argument, we prove the existence of the blow up solutions,
and state it as follows,
\begin{prop}
Let $v(t,x)\in C([0,T),\dot{H}^{\frac{1}{2}}\cap\dot{H}^{1})$ be
a solution to equation (\ref{eq:Hartree}) where $V(x)$ satisfies
the condition (\ref{connection condition}), if the initial data $v_{0}$
satisfies 
\begin{equation}
E^{V}(v_{0})=\frac{1}{2}\int|\nabla v_{0}|^{2}-\frac{1}{4}\int(V*|v_{0}|^{2})|v_{0}|^{2}<0,\label{the negative energy in hartree}
\end{equation}

then $v(t,x)$ blows up in finite time, i.e. $T<+\infty.$
\end{prop}
\begin{proof}
Since $V(x)$ is a radial function, then
\[
2\int x_{j}|v(t,x)|^{2}\partial_{x_{j}}\int V(x-y)|v(t,y)|^{2}dydx=\int\int(x_{j}-y_{j})\partial_{x_{j}}V(x-y)|v(t,x)|^{2}|v(t,y)|^{2}dydx.
\]

So we can rewrite the virial identity (\ref{virial identity in Hartree seting})
as the follows 
\begin{align*}
\frac{d^{2}}{dt^{2}}\int|x|^{2}|v|^{2} & =8\int|\nabla v|^{2}+4\int x_{j}|v|^{2}\partial_{j}(V*|v|^{2})dx\\
 & =8\int|\nabla v|^{2}+2\int\int(x_{j}-y_{j})\partial_{x_{j}}V(x-y)|v(t,x)|^{2}|v(t,y)|^{2}dydx\\
 & =16(\frac{1}{2}\int|\nabla v|^{2}+\frac{1}{8}\int\int(x_{j}-y_{j})\partial_{x_{j}}V(x-y)|v(t,x)|^{2}|v(t,y)|^{2}dydx)\\
 & :=16K^{V}(v(t)).
\end{align*}

Using the condition (\ref{connection condition}), we conclude 
\[
K^{V}(v(t))\leq E^{V}(v(t))=E^{V}(v_{0})<0,
\]
and we finish the proof. 
\end{proof}
\begin{prop}
\label{prop A1}Assume $u(t)$ is a solution in $C([0,T),\dot{H}^{\frac{1}{2}}\cap\dot{H}^{1})$
to equation (\ref{eq:Hartree}) with 
\[
\begin{cases}
u_{0}\in\dot{H}^{\frac{1}{2}}\cap\dot{H}^{1},\\
||u_{0}||_{\dot{H}^{1}}\leq M,
\end{cases}
\]
 and $V(x)\in L^{1}(\mathbb{R}^{3})\cap L^{\frac{3}{2}}(\mathbb{R}^{3}),$
then, if $u(t)$ blows up in finite time $T<\infty$ , there holds
\begin{equation}
\lim_{t\nearrow T}||u(t)||_{\dot{H}^{\frac{1}{2}}}=+\infty.
\end{equation}
\end{prop}
\begin{proof}
step 1. The energy is well-defined.

From the energy conservation law, 
\[
E^{V}(u(t))=E^{V}(u_{0})=\frac{1}{2}\int|\nabla u(t,x)|^{2}dx-\frac{1}{4}\int(V*|u|^{2})(t,x)|u(t,x)|^{2}dx.
\]

Also, with the aid of Sobolev embedding theorem in $\mathbb{R}^{3}$,
$\dot{H}^{\frac{1}{2}}\hookrightarrow L^{3}$ and $\dot{H}^{1}\hookrightarrow L^{6}$,
\begin{equation}
\int(V*|u|^{2})(t,x)|u(t,x)|^{2}dx\lesssim||(V*|u|^{2})||_{L_{x}^{3}}||u||_{L_{x}^{3}}^{2}\lesssim||V||_{L_{x}^{\frac{3}{2}}}||u||_{L_{x}^{3}}^{4}\lesssim||u||_{\dot{H}^{\frac{1}{2}}}^{4}.
\end{equation}

So $E^{V}(u(t))$ is well defined.

step 2 . Proof by contradiction.

From step 1,
\begin{equation}
E^{V}(u(t))\geq\frac{1}{2}\int|\nabla u(t,x)|^{2}dx-C||u||_{\dot{H}^{\frac{1}{2}}}^{4}.
\end{equation}

If $u(t)$ blows up at $T$, then 
\[
\lim_{t\nearrow T}||u(t)||_{\dot{H}^{1}}=+\infty.
\]

(Otherwise we can continues the solution to $[t,t+C(M)]$ with $t+C(M)>T$,
which is a contradiction.)

The fact $E^{V}(u(t))=E^{V}(u_{0})<+\infty$ shows that 
\[
\lim_{t\nearrow T}||u(t)||_{\dot{H}^{\frac{1}{2}}}=+\infty.
\]
\end{proof}
\begin{prop}
\label{prop A2}Assume the assumptions in Proposition \ref{prop A1}
hold, and $u_{0}\in H^{1}(\mathbb{R}^{3})$ , $V(x)\in L^{1}(\mathbb{R}^{3})\cap L^{\infty}(\mathbb{R}^{3})$.
Then $u(t)$ can not blow up in finite time $T<\infty$. 
\end{prop}
\begin{proof}
Recall
\begin{equation}
E^{V}(u(t))=E^{V}(u_{0})=\frac{1}{2}\int|\nabla u(t,x)|^{2}dx-\frac{1}{4}\int(V*|u|^{2})(t,x)|u(t,x)|^{2}dx<+\infty
\end{equation}

and 
\begin{align*}
\int(V*|u|^{2})(t,x)|u(t,x)|^{2}dx & \lesssim||(V*|u|^{2})||_{L_{x}^{\infty}}|||u|^{2}||_{L_{x}^{1}}\\
 & \lesssim||V||_{L_{x}^{\infty}}||u||_{L_{x}^{2}}^{4}\\
 & \lesssim||u_{0}||_{L_{x}^{2}}^{4}\\
 & \lesssim C(u_{0}),
\end{align*}

where we also use mass conservation.

The above gives 
\begin{equation}
\int|\nabla u(t,x)|^{2}dx\leq\frac{1}{2}\int(V*|u|^{2})(t,x)|u(t,x)|^{2}dx+2E^{V}(u_{0})\leq C(u_{0}),
\end{equation}

which is a contradiction to $\lim_{t\nearrow T}||u(t)||_{\dot{H}^{1}}=+\infty.$
\end{proof}
In fact, we also have a stronger characterization about the blow-up
rate of critical norm, and state it as follows:
\begin{prop}
\label{prop  A3}Under the assumptions in proposition \ref{prop A1}
and assume $u(t)$ blows up in finite time $T<+\infty$, then $\exists C=C(u_{0})\;s.t.$
\begin{equation}
||u(t)||_{\dot{H}^{1}}\geq\frac{C}{(T-t)^{\frac{1}{4}}},
\end{equation}

and 
\begin{equation}
||u(t)||_{\dot{H}^{\frac{1}{2}}}\geq\frac{C}{(T-t)^{\frac{1}{8}}}
\end{equation}

for $t$ close enough to $T$.
\end{prop}
\begin{proof}
If not, $\exists\{t_{n}\}$ with $\lim_{n\nearrow+\infty}t_{n}=T$
such that $||u(t_{n})||_{\dot{H}^{1}}\leq\frac{1}{n(T-t)^{\frac{1}{4}}}:=M_{n}.$

At the time $t_{n}$, 
\begin{equation}
\begin{cases}
i\partial_{t}u+\triangle u=-(V*|u|^{2})u. & (t,x)\in\mathbb{R}\times\mathbb{R}^{3},\\
u\mid_{t=t_{n}}=u(t_{n},x)\in\dot{H}^{1}\cap\dot{H}^{\frac{1}{2}}.
\end{cases}\label{t tends to T in Hartree}
\end{equation}

From local theory, $u(t)\in C([t_{n},t_{n}+T_{n}],\dot{H}^{\frac{1}{2}}\cap\dot{H}^{1})$
with $T_{n}\geq\frac{C}{M_{n}^{4}}\gtrsim n^{4}(T-t_{n}).$

However, $t_{n}+n^{4}(T-t_{n})>T$ for $n$ large enough, which is
a contradiction.

From the energy conservation law, 
\begin{align*}
E^{V}(u(t)) & =E^{V}(u_{0})=\frac{1}{2}\int|\nabla u(t,x)|^{2}dx-\frac{1}{4}\int(V*|u|^{2})(t,x)|u(t,x)|^{2}dx\\
 & \geq\frac{1}{2}||\nabla u(t)||_{L_{x}^{2}}^{2}-C||u(t)||_{\dot{H}^{\frac{1}{2}}}^{4}.
\end{align*}

From the (\ref{t tends to T in Hartree}) , $||u(t)||_{\dot{H}^{1}}\geq\frac{C}{(T-t)^{\frac{1}{4}}}$
, so $||u(t)||_{\dot{H}^{\frac{1}{2}}}\geq\frac{C}{(T-t)^{\frac{1}{8}}}$.
\end{proof}
\begin{rem}
In proposition \ref{prop  A3}, we have used the fact $V(x)\in L^{\frac{3}{2}}(\mathbb{R}^{3}).$
\end{rem}
\begin{rem}
Although we can give some characterizations about the blow-up rate,
we can not give a concrete example of the blow-up solution. This is
a weakness for proposition \ref{prop A1}-\ref{prop  A3}. 
\end{rem}

\section{The proof of Lemma \ref{local theory lemma C}}
\begin{proof}
We define $u_{\varepsilon}:=v_{\varepsilon}-v$ , i.e. $v_{\varepsilon}(\tau,x)=u_{\varepsilon}(\tau,x)+v(\tau,x),$
then $u_{\varepsilon}(\tau,x)$ satisfies
\begin{equation}
\begin{cases}
i\partial_{t}u_{\varepsilon}(\tau,x)+\triangle u_{\varepsilon}(\tau,x)=-(V_{\varepsilon}*|v_{\varepsilon}(\tau,x)|^{2})v_{\varepsilon}(\tau,x)+|v|^{2}v. & (\tau,x)\in\mathbb{R}\times\mathbb{R}^{3},\\
u_{\varepsilon}(\tau,x)\mid_{\tau=0}=0.
\end{cases}
\end{equation}

From the integral equation, 
\begin{equation}
u_{\varepsilon}(\tau,x)=i\int_{0}^{\tau}e^{i(\tau-s)\triangle}[(V_{\varepsilon}*|v_{\varepsilon}|^{2})v_{\varepsilon}-|v|^{2}v](s)ds,
\end{equation}

which implies 
\begin{align*}
||u_{\varepsilon}(\tau,x)||_{L_{\tau}^{\infty}([0,T],\dot{H}^{1})} & =||\int_{0}^{\tau}e^{i(\tau-s)\triangle}[(V_{\varepsilon}*|v_{\varepsilon}|^{2})v_{\varepsilon}-|v|^{2}v](s)ds||_{L_{\tau}^{\infty}([0,T],\dot{H}^{1})}\\
 & \leq||\int_{0}^{\tau}e^{i(\tau-s)\triangle}\nabla[(V_{\varepsilon}*|v_{\varepsilon}|^{2})v_{\varepsilon}-(V_{\varepsilon}*|v_{\varepsilon}|^{2})v](s)ds||_{L_{\tau}^{\infty}([0,T],L^{2})}\\
 & +||\int_{0}^{\tau}e^{i(\tau-s)\triangle}\nabla([(V_{\varepsilon}*|v_{\varepsilon}|^{2})-|v|^{2}]v(s))ds||_{L_{\tau}^{\infty}([0,T],L^{2})}\\
 & :=(I)+(II).
\end{align*}

Next we will estimate these two terms.

\begin{align*}
(I) & \lesssim||\nabla[(V_{\varepsilon}*|v_{\varepsilon}|^{2})v_{\varepsilon}-(V_{\varepsilon}*|v_{\varepsilon}|^{2})v]||_{L_{\tau}^{2}([0,T],L_{x}^{\frac{6}{5}})}\\
 & \leq||V_{\varepsilon}*(\nabla|v_{\varepsilon}|^{2})(v_{\varepsilon}(s)-v(s))||_{L_{\tau}^{2}([0,T],L_{x}^{\frac{6}{5}})}+||V_{\varepsilon}*(|v_{\varepsilon}|^{2})\nabla(v_{\varepsilon}(s)-v(s))||_{L_{\tau}^{2}([0,T],L_{x}^{\frac{6}{5}})}\\
 & \leq||V_{\varepsilon}*(\nabla|v_{\varepsilon}|^{2})||_{L_{\tau}^{\infty}([0,T],L_{x}^{\frac{3}{2}})}||(v_{\varepsilon}(s)-v(s))||_{L_{\tau}^{\infty}([0,T],L_{x}^{6})}\cdotp T^{\frac{1}{2}}\\
 & +||V_{\varepsilon}*(|v_{\varepsilon}|^{2})||_{L_{\tau}^{\infty}([0,T],L_{x}^{3})}||\nabla(v_{\varepsilon}(s)-v(s))||_{L_{\tau}^{\infty}([0,T],L_{x}^{2})}\cdotp T^{\frac{1}{2}}\\
 & \leq||\nabla v_{\varepsilon}||_{L_{\tau}^{\infty}([0,T],L_{x}^{2})}^{2}||u_{\varepsilon}(\tau,x)||_{L_{\tau}^{\infty}([0,T],\dot{H}^{1})}T^{\frac{1}{2}}\\
 & \leq(||u_{\varepsilon}(\tau,x)||_{L_{\tau}^{\infty}([0,T],\dot{H}^{1})}^{2}+||v(\tau,x)||_{L_{\tau}^{\infty}([0,T],\dot{H}^{1})}^{2})||u_{\varepsilon}(\tau,x)||_{L_{\tau}^{\infty}([0,T],\dot{H}^{1})}\cdotp T^{\frac{1}{2}}.
\end{align*}

Before we estimate part $(II)$, let us recall some elementary knowledge,

\begin{equation}
\forall f\in L^{p}(\mathbb{R}^{3}),\;p\in(1,+\infty),\;\lim_{\varepsilon\searrow0^{+}}||V_{\varepsilon}*f-f||_{L^{p}(\mathbb{R}^{3})}=0.
\end{equation}

From the local theory, $v(\tau,x)\in C([0,T],\dot{H}^{1})$, so $\forall\tilde{\delta}>0,\;\exists\delta_{1}>0,$
s.t. $\forall\tau_{1},\tau_{2}\in[0,T]$ with $|\tau_{2}-\tau_{1}|<\delta_{1}$
, $||v(\tau_{1},x)-v(\tau_{2},x)||_{\dot{H}^{1}}<\tilde{\delta}$
. Similarly, we have $||v(\tau_{1},x)-v(\tau_{2},x)||_{L^{6}}<\tilde{\delta}$
and $||v(\tau_{1},x)\cdotp\nabla v(\tau_{1},x)-v(\tau_{2},x)\cdotp\nabla v(\tau_{2},x)||_{L^{\frac{3}{2}}}<\tilde{\delta}$. 

Now we estimate the part $(II)$:
\begin{align*}
(II) & =||\int_{0}^{\tau}e^{i(\tau-s)\triangle}\nabla([(V_{\varepsilon}*|v_{\varepsilon}|^{2})-|v|^{2}]v(s))ds||_{L_{\tau}^{\infty}([0,T],L^{2})}\\
 & \leq||\nabla([(V_{\varepsilon}*|v|^{2})-|v|^{2}]v(s))||_{L_{\tau}^{2}([0,T],L_{x}^{\frac{6}{5}})}\\ &+||\nabla[(V_{\varepsilon}*(\bar{v}u_{\varepsilon}+v\bar{u}_{\varepsilon}))v(s)]||_{L_{\tau}^{2}([0,T],L_{x}^{\frac{6}{5}})}\\ &+||\nabla[(V_{\varepsilon}*|u_{\varepsilon}|^{2})v(s)]||_{L_{\tau}^{2}([0,T],L_{x}^{\frac{6}{5}})}\\
 & :=(i)+(ii)+(iii).
\end{align*}

For the term $(ii)$, 
\begin{equation}
(ii)\leq||u_{\varepsilon}||_{L_{\tau}^{\infty}([0,T],\dot{H}^{1})}||v||_{L_{\tau}^{\infty}([0,T],\dot{H}^{1})}^{2}T^{\frac{1}{2}},
\end{equation}

and the computation is also direct for $(iii)$, 
\begin{equation}
(iii)\leq||u_{\varepsilon}||_{L_{\tau}^{\infty}([0,T],\dot{H}^{1})}^{2}||v||_{L_{\tau}^{\infty}([0,T],\dot{H}^{1})}T^{\frac{1}{2}}.
\end{equation}

Finally, we estimate term $(i)$:

\begin{align*}
(i) & \leq||[(V_{\varepsilon}*|v|^{2})-|v|^{2}]\nabla v(s)||_{L_{\tau}^{2}([0,T],L_{x}^{\frac{6}{5}})}+||[(V_{\varepsilon}*\nabla|v|^{2})-\nabla|v|^{2}]v(s))||_{L_{\tau}^{2}([0,T],L_{x}^{\frac{6}{5}})}\\
 & \leq||[(V_{\varepsilon}*|v|^{2})-|v|^{2}]||_{L_{\tau}^{\infty}([0,T],L_{x}^{3})}||\nabla v(s)||_{L_{\tau}^{\infty}([0,T],L_{x}^{2})}T^{\frac{1}{2}}\\ &+||[(V_{\varepsilon}*\nabla|v|^{2})-\nabla|v|^{2}]||_{L_{\tau}^{\infty}([0,T],L_{x}^{\frac{3}{2}})}||v(s)||_{L_{\tau}^{\infty}([0,T],L_{x}^{6})}T^{\frac{1}{2}}\\
 & \leq\delta_{2}||\nabla v(s)||_{L_{\tau}^{\infty}([0,T],L_{x}^{2})}T^{\frac{1}{2}},
\end{align*}

where we used the prepared knowledge in the last inequality and $\delta_{2}$
is a small enough constant to be defined later. Summing up the above
three estimations, we conclude,

$\forall$ fixed $T\in[0,T_{max}),$ $\forall\delta_{2}>0$, $\exists\varepsilon_{1}>0$,
such that $\forall\varepsilon\in(0,\varepsilon_{1}),$ 
\begin{align}
||u_{\varepsilon}||_{L_{\tau}^{\infty}([0,T],\dot{H}^{1})} & \leq C\{||u_{\varepsilon}||_{L_{\tau}^{\infty}([0,T],\dot{H}^{1})}^{3}T^{\frac{1}{2}}+||u_{\varepsilon}||_{L_{\tau}^{\infty}([0,T],\dot{H}^{1})}^{2}||v||_{L_{\tau}^{\infty}([0,T],\dot{H}^{1})}T^{\frac{1}{2}}\label{C-6}\\
 & +||u_{\varepsilon}||_{L_{\tau}^{\infty}([0,T],\dot{H}^{1})}||v||_{L_{\tau}^{\infty}([0,T],\dot{H}^{1})}^{2}T^{\frac{1}{2}}+\delta_{2}||v||_{L_{\tau}^{\infty}([0,T],\dot{H}^{1})}^{2}T^{\frac{1}{2}}\}.\nonumber 
\end{align}

Dual to the fact $v(t)\in C([0,T],\dot{H}^{1})$, there exists $M=M(T)$
such that $||\nabla v(\tau)||_{L_{\tau}^{\infty}([0,T],L_{x}^{2})}<M$.
We can also choose $\delta_{2}$ small enough such that $\delta_{3}$
small and $\delta_{3}:=\delta_{2}M^{2}.$ Then (\ref{C-6}) can be
rewrote as 
\begin{equation}
||u_{\varepsilon}||_{L_{\tau}^{\infty}([0,T],\dot{H}^{1})}\leq C\{||u_{\varepsilon}||_{L_{\tau}^{\infty}([0,T],\dot{H}^{1})}^{3}T^{\frac{1}{2}}+||u_{\varepsilon}||_{L_{\tau}^{\infty}([0,T],\dot{H}^{1})}^{2}MT^{\frac{1}{2}}+||u_{\varepsilon}||_{L_{\tau}^{\infty}([0,T],\dot{H}^{1})}M^{2}T^{\frac{1}{2}}+\delta_{3}T^{\frac{1}{2}}\}.\label{C-7}
\end{equation}

We can divide $T$ into $N(T)$ intervals and the length of each interval
is $T_{1}$ such that $CT_{1}^{\frac{1}{2}}\ll1,\;CMT_{1}^{\frac{1}{2}}\ll\frac{1}{100}\;and\;CM^{2}T_{1}^{\frac{1}{2}}\ll\frac{1}{100}.$

From each interval, (\ref{C-7}) leads to 
\begin{equation}
||u_{\varepsilon}||_{L_{\tau}^{\infty}([\tau_{i},\tau_{i+1}],\dot{H}^{1})}\leq||u_{\varepsilon}(\tau_{i})||_{\dot{H}^{1}}+\frac{1}{100}||u_{\varepsilon}||_{L_{\tau}^{\infty}([\tau_{i},\tau_{i+1}],\dot{H}^{1})}^{3}+\frac{1}{100}||u_{\varepsilon}||_{L_{\tau}^{\infty}([\tau_{i},\tau_{i+1}],\dot{H}^{1})}^{2}+\delta_{3}.
\end{equation}

On the one hand, we choose $\delta_{3}$ such that $2^{N(T)}\delta_{3}\ll1$.
On the other hand, with the aid of standard continuity argument,
we adjust $\delta_{3}$ such that 
\begin{equation}
||u_{\varepsilon}||_{L_{\tau}^{\infty}([0,T],\dot{H}^{1})}\leq2^{N(T)}\delta_{3}\leq\delta,
\end{equation}

which ends the proof.
\end{proof}

\end{document}